\documentclass[12pt,reqno]{amsart} 
\usepackage{amsmath,amssymb,amsthm,mathrsfs}
\usepackage[english]{babel}
\usepackage[utf8]{inputenc}
\usepackage{xcolor}

\newtheorem{theorem}{Theorem}[section]

\theoremstyle{definition}
\newtheorem{definition}[theorem]{Definition}

\newtheorem{remark}[theorem]{Remark}

\usepackage{color}
\usepackage{multirow} % para las tablas

\usepackage{graphicx}
\usepackage[left=3cm,right=3cm]{geometry}

\def\r{\mathbb R}

\def\s{\mathbb S}

\def\H{\mathcal{H}}
\def\h{\mathbb H}

\def\h{\mathbb H}

\begin{document}

\title{The class of grim reapers in $\mathbb{H}^2\times\mathbb{R}$}
\author{Antonio Bueno, Rafael L\'opez}
\address{ Departamento de Ciencias\\  Centro Universitario de la Defensa de San Javier. 30729 Santiago de la Ribera, Spain}
\email{antonio.bueno@cud.upct.es}
\address{ Departamento de Geometr\'{\i}a y Topolog\'{\i}a\\  Universidad de Granada. 18071 Granada, Spain}
\email{rcamino@ugr.es}

\begin{abstract}
We study translators of the mean curvature flow in the product space $\h^2\times\r$. In $\h^2\times\r$ there are three types of translations: vertical translations  due to the   factor $\r$ and    parabolic and hyperbolic translations from $\h^2$.  A   grim reaper in $\h^2\times\r$ is  a translator invariant by a one-parameter group of translations.  The variety of translators and translations in $\h^2\times\r$ makes that the family of grim reapers particularly rich. In this paper we give a full classification of the grim reapers of $\h^2\times\r$ with a description of their geometric properties.    In some cases, we obtain explicit parametrizations of the surfaces.
 \end{abstract}


\maketitle
 
\setcounter{tocdepth}{1}
 \tableofcontents

 %%%%%%%%%%%%%%%%%%%%%%%%%%%%%%
\section{Introduction}

In the last decades, the theory of mean curvature flow (MCF for short) in Euclidean space $\r^3$ is an active and fruitful field of research, see e.g. \cite{co,ec,hu,ma} for an outline of the development of this theory. Specifically,  let $\Sigma$ be an orientable smooth surface and let $\psi:\Sigma\to \r^3$ be an isometric immersion. Consider a   variation of $\Sigma$ given by a one-parameter smooth family of immersions $\psi_t:\Sigma\to\r^3$, $t\in[0,T)$, $T>0$, where   $\psi_0=\psi$. We say that $\{\psi_t:t\in [0,T)\}$ evolves by mean curvature flow  if
$$
\left\{
\begin{array}{cl}
\displaystyle{\frac{\partial\psi_t}{\partial t}}&=H(\psi_t)N(\psi_t),\\
\psi_0&=\psi,
\end{array}
\right.
$$
where $H(\psi_t)$ and $N(\psi_t)$ are the mean curvature and the unit normal  of $\psi_t$, respectively.  Of special interest are those surfaces   that are self-similar solutions to the MCF in the sense that the surface moves under  a combination of dilations and isometries of $\r^3$. An example of self-similar solutions  that has received a great interest   is when $\Sigma$ evolves purely by   translations of the ambient space $\r^3$. Specifically, fix a unit vector   $\textbf{v}\in\r^3$. A surface   $\Sigma$ is said to be a translator  if the MCF evolves by
$$
\psi_t(p)=\psi(p)+t\textbf{v}.
$$
Then, $H(\psi_t)$ and  $N(\psi_t) $ coincide with the mean curvature $H$ and the unit normal vector of $\Sigma$ respectively. The evolution equation writes then as
\begin{equation}\label{eq1}
H=\langle N,\textbf{v}\rangle.
\end{equation}
Translators also are important  because they arise in the singularity theory of MCF after a proper rescaling around type-II singularities \cite{hs}.  Without aiming to collect all the bibliography, we refer the reader to \cite{css,sx} and references therein.  Examples of translators are those one which are invariant by  isometries of $\r^3$. The simplest case is then the translator is   invariant in a spatial direction of $\r^3$. In such a case and besides the plane, the translator is called a grim reaper. A grim reaper is a cylindrical surface erected over a planar curve which   is a solution of the curve shortening flow, the analogy of the MCF in dimension $2$.

A space where it is of interest the study of the MCF is the   hyperbolic space $\h^3$. Motivated by a work of Huisken in arbitrary Riemannian manifolds \cite{hu}, Cabezas-Rivas and Miquel considered the MCF in $\h^3$ assuming that the volume is preserved along the flow \cite{cm1,cm2}. If the initial surface is  convex, an interesting question is how the convexity is preserved or not along the flow and what is the limit. See also   \cite{an,aw,cm3,ji,ge,wa}. In contrast to the Euclidean case, the study of the analogs of  translators of $\h^3$ has been unexplored until the last year, when several authors, including the present ones, approached this problem. Following the motivation from the Euclidean space, for the definition of translator, it is necessary   to consider the two types of translations of   $\h^3$, obtaining two types of translators \cite{bl1,li}. Each type of translations of $\h^3$ has associated a Killing vector field $X\in\mathfrak{X}(\h^3)$ which is placed in \eqref{eq1} instead of the vector $\textbf{v}$. Then a translator with respect to $X$ is characterized by the equation $H=\langle N,X\rangle$. In $\h^3$, conformal vector fields can be also used in this equation. For some of these conformal vector fields, the corresponding   surfaces satisfying this equation are the analogs   of self-shrinkers and self-expanders of $\r^3$ \cite{bl2,mr}.  In a more general setting, it was in \cite{alr} where self-similar solutions to the MCF were investigated  under the presence of a conformal vector field. Some results in $\h^3$ were achieved in \cite{alr} by viewing this space as a warped product.

The product space $\h^2\times\r$ is other space  where the MCF is interesting. The importance of this space is because it is one of the eight model geometries of Thurston \cite{th}. The space $\h^2\times\r$ is defined as the Riemannian product of $\h^2$ with the Euclidean real line $\r$, endowed with the usual product metric.    The product structure defines a natural projection $\pi:\h^2\times\r\to\h^2$  whose fibers   are vertical geodesic.

Following with the motivation from Euclidean and hyperbolic spaces, the definition of translator of the MCF in $\h^2\times\r$ consists in replacing   the vector ${\bf v}$ of the right hand-side of  Eq. \eqref{eq1}  by a Killing vector field whose isometries are translations of  $\h^2\times\r$. In order to make the definitions, we will consider  the   upper half-plane model of $\h^2$, that is, the half-plane $\r^2_+=\{(x,y)\colon y>0\}$ endowed with the metric $\langle,\rangle=\frac{1}{y^2}\langle,\rangle_e$, where $\langle,\rangle_e$ stands for the  Euclidean metric in $\r^2_+$. The ideal boundary $\h^2_\infty$ is the one-point compactification of the boundary line $y=0$, and consequently, we have the ideal boundary of $\h^2\times\r$, namely, $\h^2_\infty\times\r$.

The product structure makes readily available three Killing vector fields which generate translations of $\h^2\times\r$. 
\begin{enumerate}
\item The vector field   $\partial_z$. The corresponding isometries are translations along the fibers. They are called vertical translations. From the Euclidean viewpoint, they are  Euclidean translations $(x,y,z)\mapsto (x,y,z+t)$, $t\in\r$.
\item The vector field $\partial_x$. Then $\partial_x$ induces   parabolic translations in $\h^2$. In our model for $\h^2\times\r$, they are   translations in the direction of the $x$-axis, $(x,y,z)\mapsto (x+t,y,z)$, $t\in\r$. 
\item The vector field $x\partial_x+y\partial_y$. The isometries of this Killing vector field are hyperbolic isometries. From the Euclidean viewpoint, they are dilations in $\r^2_{+}$, $(x,y,z)\mapsto (e^t x, e^ty,z)$, $t\in\r$. 
\end{enumerate}
Coming back to Eq. \eqref{eq1}, it is natural to  assume that the role of $\textbf{v}$    can be played by each one of the above vector fields. This motivates the following definition.  

\begin{definition}\label{def1}
Let $\Sigma$ be an orientable immersed surface  in  $\h^2\times\r$. Let $H$ and $N$ be the mean curvature and the unit normal of $\Sigma$. Then $\Sigma$ is called 
\begin{enumerate}
\item a v-translator if $H$ satisfies
\begin{equation}\label{eq21}
H=\langle N,\partial_z\rangle,
\end{equation}
\item a p-translator if $H$ satisfies
\begin{equation}\label{eq22}
H=\langle N,\partial_x\rangle,
\end{equation}
\item a h-translator if $H$ satisfies
\begin{equation}\label{eq23}
H=\langle N,x\partial_x+y\partial_y\rangle,
\end{equation}
\end{enumerate}
\end{definition}

The theory of v-translators in $\h^2\times\r$ was initiated by the first author in \cite{bue1,bue2} obtaining the classification of rotational v-translators about a vertical geodesic. The asymptotic behavior at infinity and   several uniqueness and non-existence results were also obtained. Independently,  Lira and Mart\'{\i}n \cite{lima} investigated v-translators in the general setting of Riemannian products $M^2\times\r$, in particular, when $M^2=\h^2$. In that paper, they studied rotational v-translators and also v-translators invariant by hyperbolic and parabolic translations,  but a specific geometric description of these surfaces was missing. More recently, and the same time of the present paper, Lima and Pipoli have studied  v-translators in arbitrary dimension $\h^n\times\r$   \cite{lipi}. They have extended the notion of v-translators to the so-called translators of the $r$-mean curvature flow, where their Killing vector field agree with our Killing vector field $\partial_z$. In particular, they also named parabolic and hyperbolic grim reapers to those v-translators invariant by parabolic and hyperbolic translations. 

Besides the theory of v-translators, in Def. \ref{def1} we have introduced two new notions of translators that have not been previously considered in the literature of $\h^2\times\r$. Although the vector field $\partial_z$ is a relevant Killing vector field, the Killing vector fields $\partial_x$ and $x\partial_x+y\partial_y$ coming from the factor $\h^2$ are equally important from the MCF viewpoint. So p-translators and h-translators are self-similar solutions of the MCF in $\h^2\times\r$ that open new perspectives in this field.

As a first contribution to this unexplored theory of p-translators and h-translators, we give a natural definition of  grim reapers in $\h^2\times\r$ that extends that of Euclidean space $\r^3$ and is also consistent to the one considered in \cite{lipi,lima}. Recall  that a grim reaper in $\r^3$ is a translator invariant by  translations along a spatial direction. The translator is associated with a Killing vector field of $\r^3$, which in this case, it is identified with  $\textbf{v}$. In our situation, a grim reaper in $\h^2\times\r$ is defined as a translator invariant by a one-parameter group of  translations of the ambient space. Consequently, we have to distinguish the one-parameter group of translations that will define these grim reapers, and also take into account the vector field that defines the translator equation. The next definition clarifies this distinction.

\begin{definition}\label{def2}
Let $\delta\in\{\mathrm{v,p,h}\}$. A vertical (resp. parabolic, hyperbolic) $\delta$-grim reaper is a $\delta$-translator invariant by the one-parameter group of vertical (resp. parabolic, hyperbolic) translations of $\h^2\times\r$.
\end{definition}

Since we have three different types of translations and three different types of Killing vector fields, we have a total of nine different types of grim reapers. In this paper and for some cases, we are able to integrate the ODE fulfilled by the coordinate functions of the corresponding generating curve, obtaining explicit examples of translators in $\h^2\times\r$.

Another vector fields of interest that can be considered in the translator equation are the conformal vector fields. This gives a new family of translators which has not been previously considered in $\h^2\times\r$. So, we replace the Killing vector fields of Def. \ref{def1} by  particular conformal vector fields. Indeed,  motivated by the recent works \cite{bl2,mr}, we can also define self-similar solutions to the MCF in $\h^2\times\r$ that extend the notions of  self-shrinkers and self-expanders. For this, we consider the conformal vectors fields $\partial_y$ and $-\partial_y$.  
\begin{definition}\label{def3}
A $c_+$-grim reaper is a surface in $\h^2\times\r$ such that its mean curvature $H$ satisfies
\begin{equation}\label{c+}
H=\langle N,\partial_y\rangle.
\end{equation}
Similarly, a $c_-$-grim reaper is a surface that satisfies
\begin{equation}\label{c-}
H=-\langle N,\partial_y\rangle.
\end{equation}
\end{definition}
These six new examples of grim reapers together with the nine grim reapers of Def. \ref{def2} add a total of fifteen types of grim reapers in $\h^2\times\r$. We will adopt the convention that when use {\it translators} we mean all the three types of translators; {\it grim reapers} assigns all types of grim reapers; {\it vertical grim reapers} are the vertical $\delta$-grim reapers, and so on. 

In this paper we exhibit a full classification of every type of grim reaper. A case that can be considered as trivial is when the right hand-side of the translator equation $H=\langle N,X\rangle$ is zero, since the surface is minimal ($H=0$). This occurs for vertical v-grim reapers, parabolic p-grim reapers and hyperbolic h-grim reapers (Thm. \ref{t0}). In Table \ref{table1} we summarize the profile curves of all the grim reapers. Each row corresponds to a vector field, being the first three the Killing vector fields and the last two the conformal vector fields. As for the rows, they indicate the one-parameter group of translations under which the grim reaper is invariant. A direct consequence is the following:

\begin{theorem}
All grim reapers of $\h^2\times\r$ are embedded surfaces.
\end{theorem}

Another consequence of this classification is that for Killing vector fields, parabolic grim reapers have explicity parametrizations (Thms. \ref{t1} and \ref{t62}). 

\begin{theorem}
Parabolic $\delta$-grim reapers, with $\delta\in\{v,p,h\}$  have explicit parametrizations in terms of  simple functions.
\end{theorem}

\begin{table}[ht]
\centering
\begin{tabular}{c|ccc}
&Parabolic&Hyperbolic&Vertical\\
&$yz$-plane& $\s^1_+\times\r$&$xy$-plane\\
%\hline  
%& $(t,y(s),z(s))$ & $(e^t\cos r(s),e^t\sin r(s),z(s))$ & $(x(s),y(s),t)$\\
\hline
$\partial_z$& \includegraphics[width=.1\textwidth]{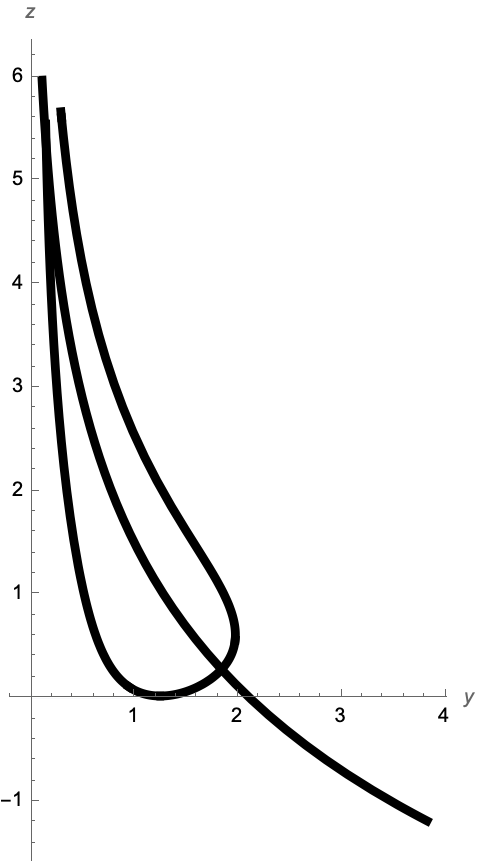} &\includegraphics[width=.15\textwidth]{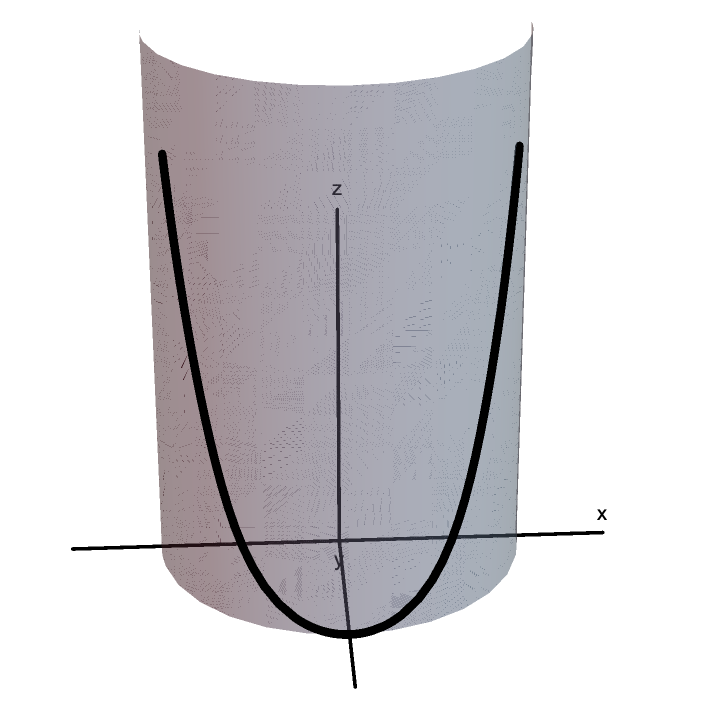} & \includegraphics[width=.2\textwidth]{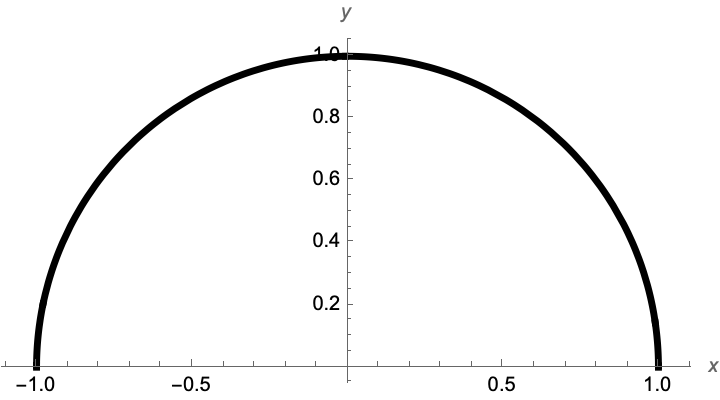} \\
$\partial_x$& \includegraphics[width=.08\textwidth]{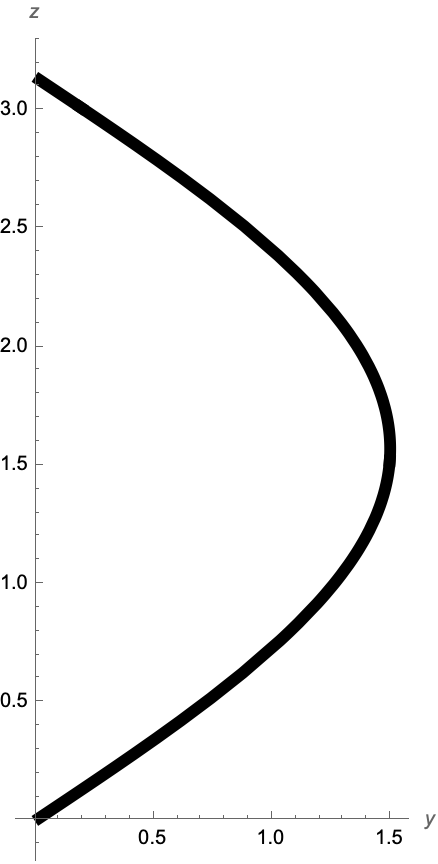} & \includegraphics[width=.15\textwidth]{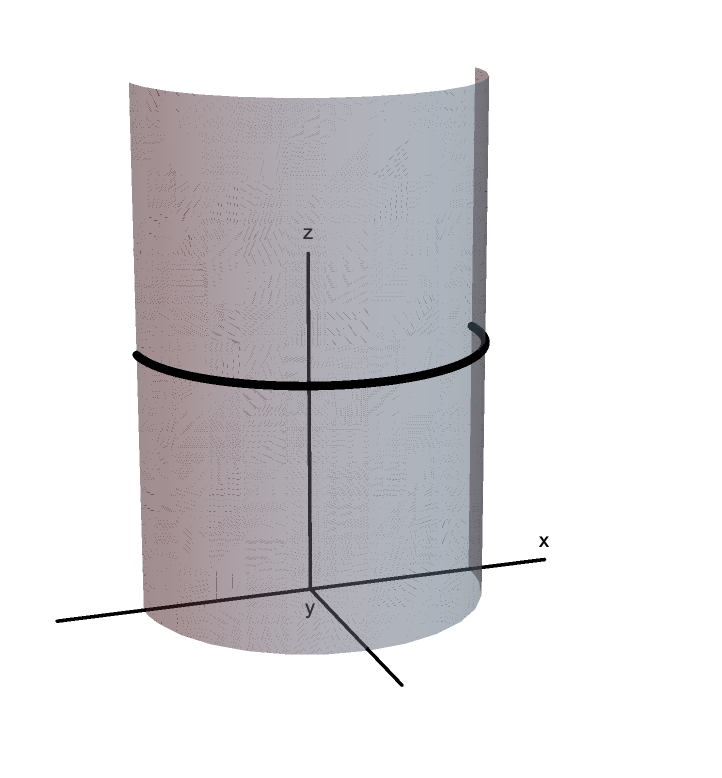} & \includegraphics[width=.2\textwidth]{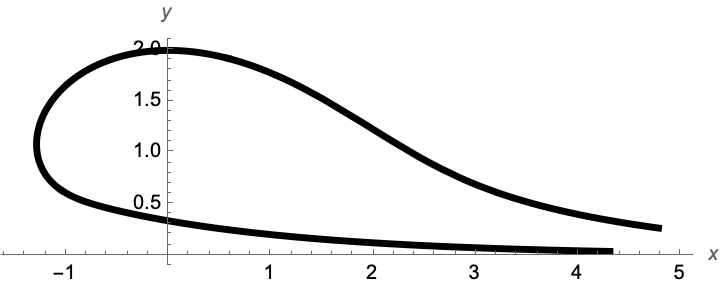}\\
$x\partial_x+y\partial_y$&  \includegraphics[width=.2\textwidth]{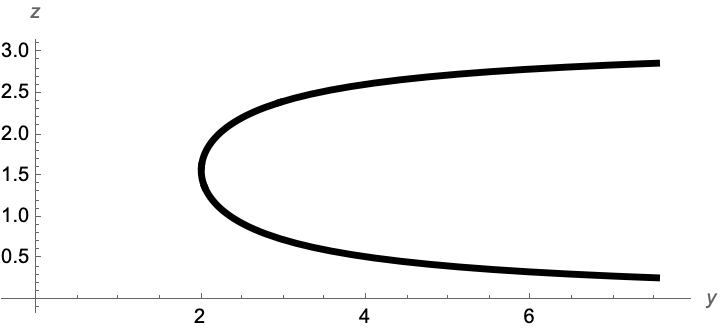} & \includegraphics[width=.15\textwidth]{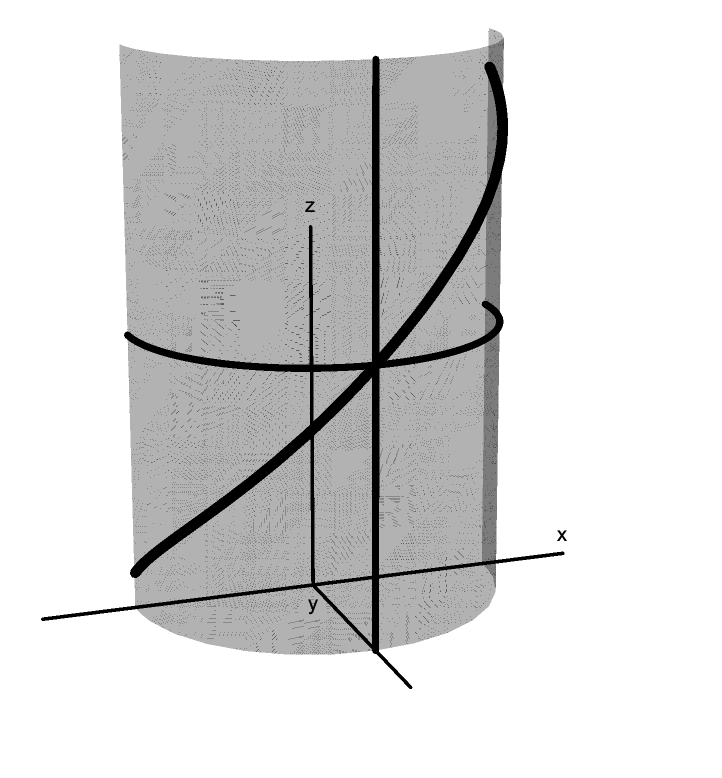} &  \includegraphics[width=.2\textwidth]{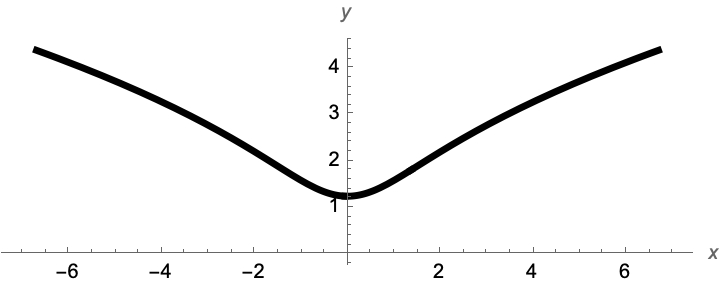}\\
$\partial_y$& \includegraphics[width=.08\textwidth]{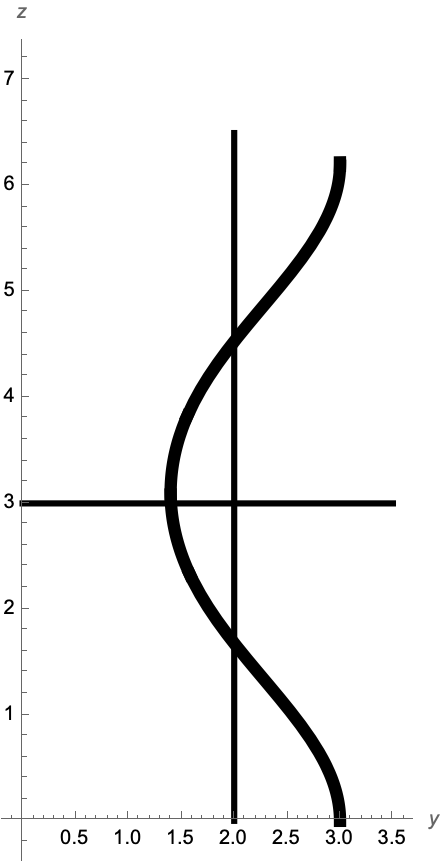} & \includegraphics[width=.15\textwidth]{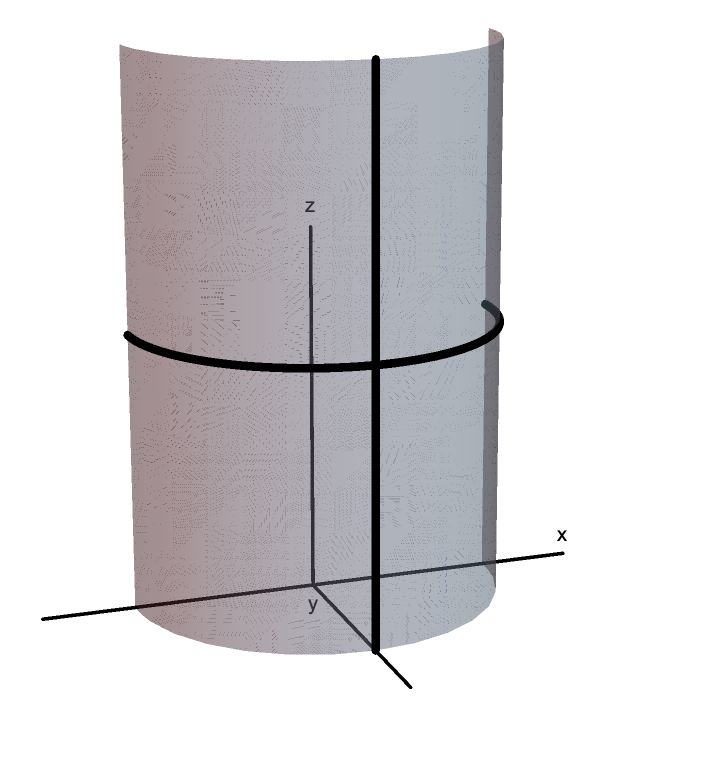} & \includegraphics[width=.2\textwidth]{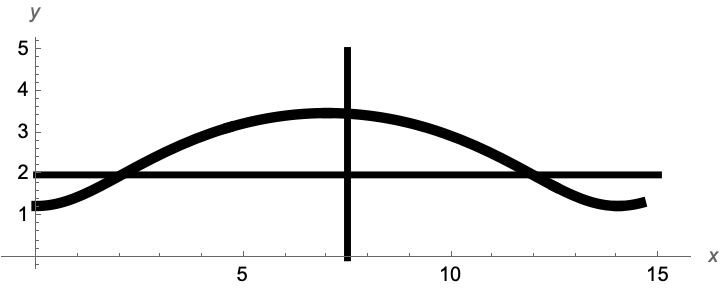}\\
$-\partial_y$& \includegraphics[width=.15\textwidth]{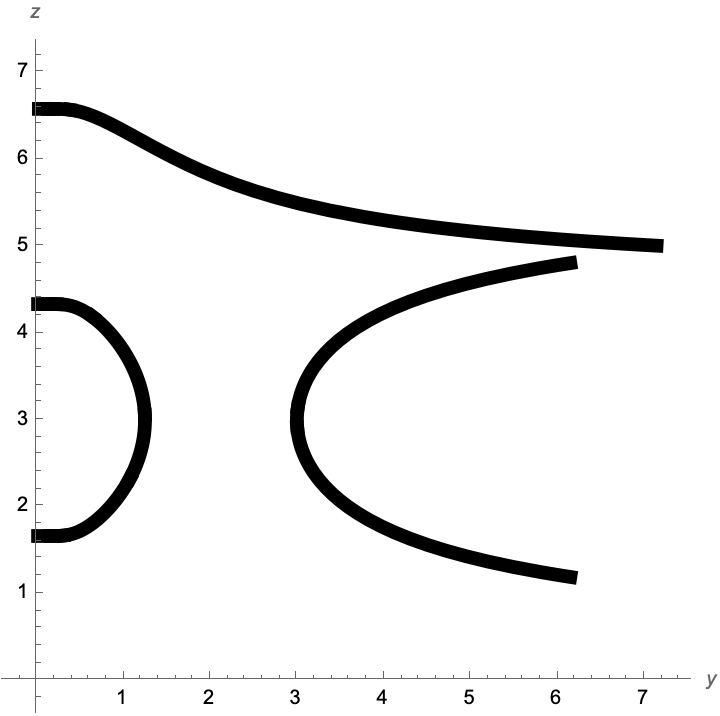} & \includegraphics[width=.15\textwidth]{perfilchyperbolic.png} & \includegraphics[width=.2\textwidth]{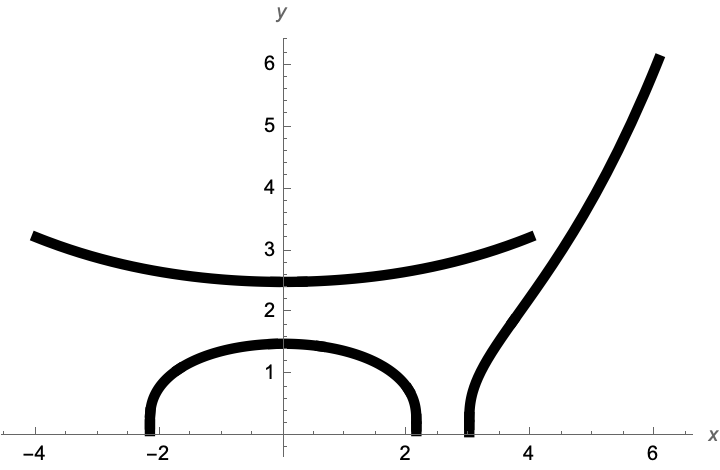}
\end{tabular}
\caption{Generating curves of the grim reapers in $\mathbb{H}^2\times\mathbb{R}$}\label{table1}
\end{table}

The organization of the paper is as follows. Section \ref{sec2} of preliminaries shows the parametrizations of the surfaces invariant by translations of $\h^2\times\r$. With these parametrizations, we compute the mean curvature $H$ and the normal vector $N$. This is necessary in order to express the translator equation $H=\langle N,X\rangle$ as ODEs fulfilled by the coordinate functions of their generating curves. As a consequence of these computations, in Sect. \ref{sec3} we revisit the minimal surfaces invariant by a one parameter group of translations of $\h^2\times\r$, since they appear as trivial cases for some grim reapers (Thm. \ref{t0}). We also classify p-translators and h-translators of rotational type (Thm. \ref{t02}), proving that the only such examples are horizontal slices. In Sects. \ref{sec4}, \ref{sec5} and \ref{sec6} we classify the v-grim reapers, p-grim reapers and h-grim reapers, respectively. Subsequently, in Sect. \ref{sec7} we study the grim reapers defined by   conformal vector fields  of $\h^2\times\r$ (Def. \ref{def3}).

%%%%%%%%%%%%%%%%%%%%%%%%%%%%%
\section{Preliminaries}\label{sec2}
%%%%%%%%%%%%%%%%%%%%%%%%%%%%%

The purpose of this section is to investigate the translator equation $H=\langle N,X\rangle$ for each of the grim reapers in $\h^2\times\r$, as defined in Definitions \ref{def2} and \ref{def3}. For that matter, we begin by recalling the three types of one-parameter group of translations. Then, we introduce suitable parametrizations for each grim reaper and compute their mean curvatures $H$ and their unit normals $N$. With all these computations we will stand in position to study individually each of Eqns. \eqref{eq21}--\eqref{c-}.

%In this section, we compute each of the terms of equations \eqref{eq21}, \eqref{eq22} and \eqref{eq23}. First we recall the three types of one-parameter group of translations:

Starting with the translations of $\h^2\times\r$, each one belongs to any of the following one-parameter group of isometries:
\begin{enumerate}
\item   Vertical translations. These are generated by the Killing vector field $\partial_z$. Each vertical translation is an Euclidean vertical traslation given by 
$$
\mathcal{V}_t:\h^2\times\r\to \h^2\times\r,\qquad\mathcal{V}_t(x,y,z)=(x,y,z+t),\quad t\in\r.
$$
The flow of a point $(x_0,y_0,z_0)$ under these translations is the vertical  straight-line through $(x_0,y_0,z_0)$. This curve is a geodesic of $\h^2\times\r$ and  its curvature is $0$.
The one-parameter group is $\mathcal{V}=\{\mathcal{V}_t:t\in\r\}$.

\item Parabolic translations. They   fix a point of $\h^2_\infty$. Assuming that  this point is $\infty$, they are       Euclidean translations
$$
\mathcal{P}_t:\h^2\times\r\to \h^2\times\r,\qquad\mathcal{P}_t(x,y,z)=(x+t,y,z),\quad t\in\r.
$$
These are generated by the Killing vector field $\partial_x$. The flow of a point $(x_0,y_0,z_0)$ under these translations is   the horizontal straight-line through $(x_0,y_0,z_0)$ and parallel to the $x$-direction, whose projection on $\h^2$ is a horocycle. The one-parameter group is $\mathcal{P}=\{\mathcal{P}_t:t\in\r\}$.
\item Hyperbolic translations. They fix two points  $\h^2_\infty$. Assuming that both points are $(0,0)$ and $\infty$,  these translations are the Euclidean homotheties  
$$
\H_t:\h^2\times\r \to\h^2\times\r,\qquad\H_t(x,y,z)=(e^t x,e^ty,z),\quad t\in\r.
$$
These are generated by the Killing vector field $x\partial_x+y\partial_y$. The flow of a point $(x_0,y_0,z_0)$  is the horizontal straight-line through $(x_0,y_0,z_0)$ and   $(x_0,y_0,0)$, whose projection on $\h^2$ is an equidistant line from the origin of $\r^2_+$.   The one-parameter group is $\mathcal{H}=\{\mathcal{H}_t:t\in\r\}$.
\end{enumerate}
Using the three types of translations of $\h^2\times\r$, we have the corresponding types of grim reapers given in Def. \ref{def2}. 

Our next goal is to introduce suitable parametrizations $\Psi=\Psi(s,t)$ of  grim reapers in order to compute their mean curvatures $H$ and unit normals $N$. This is necessary for computing the translator equations \eqref{eq21}, \eqref{eq22} and \eqref{eq23}. For that matter, we begin by considering surfaces invariant under the aforementioned translations of   $\r^2_+\times\r$ without assuming hypothesis on its mean curvature. 

\begin{enumerate}
\item Vertical surfaces.  The surface is a ruled surface with vertical rulings. Thus 
\begin{equation}\label{p1}
\Psi(s,t)=(x(s),y(s),t),\quad s\in I\subset\r,\ t\in\r.
\end{equation}
The generating curve is $\alpha(s)=(x(s),y(s))$ contained in $\h^2\times\{0\}$.
\item Parabolic surfaces. The surface is a ruled surface with  rulings parallel to the $x$-axis of $\r^2_+\times\r$. Thus 
\begin{equation}\label{p2}
\Psi(s,t)=(t,y(s),z(s)),\quad s\in I\subset\r,\ t\in\r.
\end{equation}
The generating curve is $\alpha(s)=(y(s),z(s))$ contained in the $yz$-plane.
\item Hyperbolic  sufaces. The surface is a ruled surface whose rulings are horizontal half-lines starting from the $z$-axis.  Hence, there exists $r=r(s)$ such that $x(s)=\cos r(s)$, $ y(s)=\sin r(s)$, where $r(s)\in(0,\pi)$ because of the condition $y>0$. A parametrization for $\Sigma$ is
\begin{equation}\label{p3}
\Psi(s,t)=(e^t\cos r(s),e^t\sin r(s),z(s)),\qquad s\in I\subset\r,\ t\in\r.
\end{equation}
\end{enumerate}

We calculate the mean curvature $H$ and the unit normal $N$ of each of the above three types of surfaces. We will employ local coordinates $(x,y,z)$ for  $\h^2\times\r$. A  global orthonormal frame of $\h^2\times\r$ is  $\{E_1,E_2,E_3\}$   where 
$$E_1=y\partial_x,\quad E_2=y\partial_y,\quad E_3=\partial_z.$$
 The Levi-Civita connection $\overline{\nabla}$ of   $\h^2\times\r$ is given by 
$$\overline{\nabla}_{E_1}E_1=E_2,\quad \overline{\nabla}_{E_1}E_2=-E_1$$
and $\overline{\nabla}_{E_i}E_j=0$ in the rest of the cases.  The mean curvature $H$ of a surface $\Sigma$ parametrized by  $\Psi=\Psi(s,t)$     is given by the formula 
$$H= \frac{b_{11}g_{22}-2b_{12}g_{12}+b_{22}g_{11}}{2(g_{11}g_{22}-g_{12}^2)},$$
where  
$$g_{11}=\langle \Psi_s,\Psi_s\rangle, \quad g_{12}=\langle \Psi_s,\Psi_t\rangle\quad g_{22}=\langle \Psi_t,\Psi_t\rangle,$$
and 
$$b_{11}=\langle N,\overline{\nabla}_{\Psi_s}\Psi_s\rangle, \quad b_{12}=\langle N,\overline{\nabla}_{\Psi_s}\Psi_t\rangle, \quad b_{22}=\langle N,\overline{\nabla}_{\Psi_t}\Psi_t\rangle.$$
Now we find   $H$ and $N$  of the above surfaces.  
 \begin{enumerate}
\item  Vertical surfaces. Consider a surface $\Sigma$ parametrized by \eqref{p1}. Then $\Sigma$ is a vertical cylinder erected over the curve $\alpha(s)= (x(s),y(s))$. A direct argument show that the mean curvature of $\Sigma$ is $H=\kappa_g/2$, where $\kappa_g$ is the geodesic curvature of the curve $\alpha$. This is because the vertical rulings are geodesics of $\Sigma$. However,  we compute $H$ by completeness.  We parametrize $\alpha$ by the hyperbolic arc-length  
\begin{equation}\label{sv}
\left\{\begin{split}
x'(s)&=y(s)\cos\theta(s)\\
y'(s)&=y(s)\sin\theta(s),
\end{split}\right.
\end{equation}
 for some smooth function $\theta=\theta(s)$. Then
$$\Psi_s=\cos\theta E_1+\sin\theta E_2,\quad \Psi_t=E_3,$$
and
\begin{equation}\label{normal-v}
N=\sin\theta E_1-\cos\theta E_2.
\end{equation}
Then $g_{ij}=\delta_{ij}$. Moreover,  
\begin{equation*}
\begin{split}
\overline{\nabla}_{\Psi_s}\Psi_s&=-\sin\theta(\theta'+\cos\theta)  E_1+\cos\theta(\theta'+\cos\theta)E_2,\\
\overline{\nabla}_{\Psi_s}\Psi_t&=\overline{\nabla}_{\Psi_t}\Psi_t=0.
\end{split}
\end{equation*}
Hence, $b_{11}=-\theta'-\cos\theta$ and $b_{12}=b_{22}=0$. Thus 
\begin{equation}\label{mean-v}
H=-\frac{\theta'+\cos\theta}{2}.
\end{equation}

\item Parabolic surfaces. The surface is parametrized by \eqref{p2}. Consider the curve $\alpha(s)=(y(s),z(s))$ parametrized by
\begin{equation}\label{sp}
\left\{\begin{split}
y'(s)&=y(s)\cos\theta(s)\\
z'(s)&= \sin\theta(s),
\end{split}
\right.
\end{equation}
 for some smooth function $\theta=\theta(s)$. Then
$$\Psi_s=\cos\theta E_2+\sin\theta E_3,\quad \Psi_t=\frac{1}{y}E_1$$
and
\begin{equation}\label{normal-p}
 N=\sin\theta E_2-\cos\theta E_3.
 \end{equation}
Thus $g_{11}=1$, $g_{12}=0$ and $g_{22}=1/y^2$.  Moreover,  
\begin{equation*}
\begin{split}
\overline{\nabla}_{\Psi_s}\Psi_s&=\theta' (-\sin\theta   E_2+\cos\theta   E_3),\\
\overline{\nabla}_{\Psi_s}\Psi_t&=-\cos\theta E_1,\\
\overline{\nabla}_{\Psi_t}\Psi_t&=\frac{1}{y^2}E_2.
\end{split}
\end{equation*}
Hence, $b_{11}=-\theta'$, $b_{12}=0$ and  $b_{22}= \sin\theta/y^2$. Thus 
\begin{equation}\label{mean-p}
H=-\frac{\theta'-\sin\theta}{2}.
\end{equation}

\item Hyperbolic surfaces. Suppose that the surface is parametrized by  \eqref{p3}.  Consider the curve $\alpha(s)=(\cos r(s),\sin r(s),z(s))$ parametrized by arc-length, that is,   
\begin{equation}\label{sh}
\left\{\begin{split}
r'(s)&=\sin r(s)\cos\rho(s)\\
z'(s)&=\sin\rho(s),
\end{split}\right.
\end{equation}
for some function $\rho=\rho(s)$. Then
\begin{equation*}
\begin{split}
\Psi_s&=-\sin r\cos\rho  E_1+\cos r\cos\rho E_2+\sin\rho E_3,\\
 \Psi_t&=\cot r E_1,
 \end{split}
\end{equation*}
and
\begin{equation}\label{normal-h}
 N=-\sin r\sin\rho E_1+\cos r\sin\rho E_2-\cos \rho E_3.
 \end{equation}
Thus $g_{11}=1$, $g_{12}=0$ and $g_{22}=1/(\sin  r)^2$.  Moreover,  
\begin{equation*}
\begin{split}
\overline{\nabla}_{\Psi_s}\Psi_s=&(\rho '\sin r  \sin \rho +\sin r \cos r (\cos\rho)^2-r' \cos r \cos \rho ) E_1\\
&((\sin r  \cos \rho )^2-r' \sin r \cos \rho -\rho '\cos r  \sin \rho ) E_2+\rho ' \cos \rho E_3\\
\overline{\nabla}_{\Psi_s}\Psi_t=&(\sin r \cos \rho-r' (\csc  r)^2 ) E_1-\cos r \cos  \rho  E_2,\\
\overline{\nabla}_{\Psi_t}\Psi_t=&-\cot r E_1+(\cot r)^2 E_2.
\end{split}
\end{equation*}
Hence, 
\begin{equation*}
\begin{split}
b_{11}&=\rho',\\
b_{12}&=\sin \rho  \left(r' \csc r-\cos \rho \right),\\
b_{22}&=\cot r \csc r\sin \rho.
\end{split}
\end{equation*}
 Thus 
\begin{equation}\label{mean-h}
H=-\frac{1}{2} \left(  \rho ' -\cos r \sin  \rho\right).
\end{equation}

\end{enumerate}

%%%%
\section{Minimal surfaces and rotational translators}\label{sec3}
%%%%%%%%%%%%%%%%%%%%%
An immediate consequence of the computations of the previous section is the following partial classification   which states that three of the fifteen types of grim reapers are trivial because they are minimal surfaces. 

\begin{theorem} \label{t0}
Any vertical v-grim reaper, parabolic p-grim reaper and hyperbolic h-grim reaper is a minimal surface.
\end{theorem}

\begin{proof} It is immediate because in all cases, the normal vector of the surface is orthogonal to the Killing vector field. 
\end{proof}

Minimal surfaces in $\h^2\times\r$   invariant by vertical, parabolic and hyperbolic translation are known \cite{on}. For completeness, we revisit this kind of surfaces. See Fig. \ref{figminimal}.
\begin{enumerate}
\item Vertical minimal surfaces. The generating curve $\alpha$ is a geodesic of $\h^2$ because the  rulings are also geodesics and the surface is minimal. This means that $\alpha$ is a half-circle of $\h^2$ centered at the line of equation $y=0$. We can also solve $H=0$ from \eqref{sv} and \eqref{mean-v}. Then $\alpha(s)=(x(s),y(s))=(c_1 \tanh (s)+c_2,  c_1 /\cosh(s))$, $c_1>0$, $c_1,c_2\in\r$. This curve is a half-circle of radius $c_1$ in the $xy$-plane centered at $(c_2,0)$.   

\item Parabolic minimal surfaces. Now \eqref{mean-p} yields $\theta(s)=\sin\theta(s)$. The solution of \eqref{sp} is $\alpha(s)=(e^s,z_0)$, $s>0$, $z_0\in\r$ and 
$\alpha(s)=(y(s),z(s))=(\frac{c_1}{\cosh (s)}, c_2+2 \tan ^{-1}e^s)$, $c_1>0$, $c_1,c_2\in\r$.  The first curve corresponds with the slice $\h2\times\{z_0\}$, $z_0\in\r$.

\item Hyperbolic minimal surfaces. Now \eqref{mean-h} implies that $\rho'=\cos r\sin\rho$. The vertical plane $x=0$ and the horizontal slices are hyperbolic minimal surfaces. In this case it is not possible to get an explicit integration of \eqref{sh} in all its generality. Besides these trivial examples, we assume that $r'$ and $\rho'$ are not null. Dividing $\rho'=\cos r\sin\rho$ by $r'=\sin r\cos\rho$ and after some manipulations and integration yields
$$
\sin r=c\sin\rho,\qquad c>0.
$$
We refer the reader to \cite{on} for further details.
\end{enumerate}
 
 \begin{figure}[htbp]
\includegraphics[width=.3\textwidth]{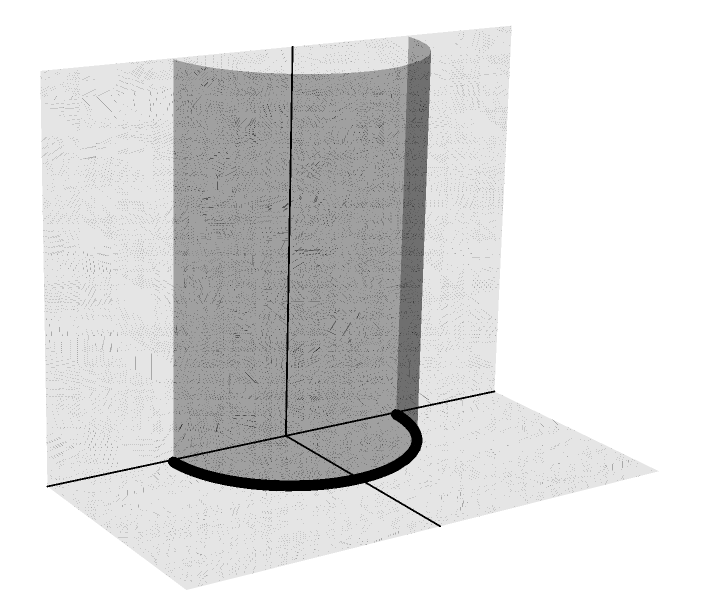}\includegraphics[width=.25\textwidth]{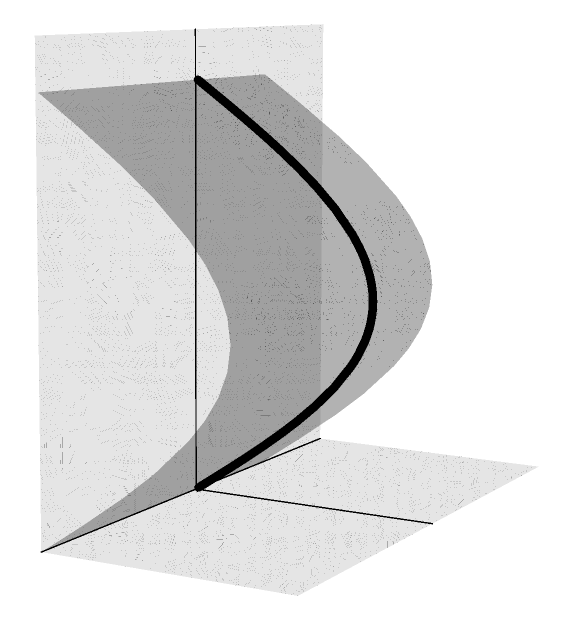}
\includegraphics[width=.3\textwidth]{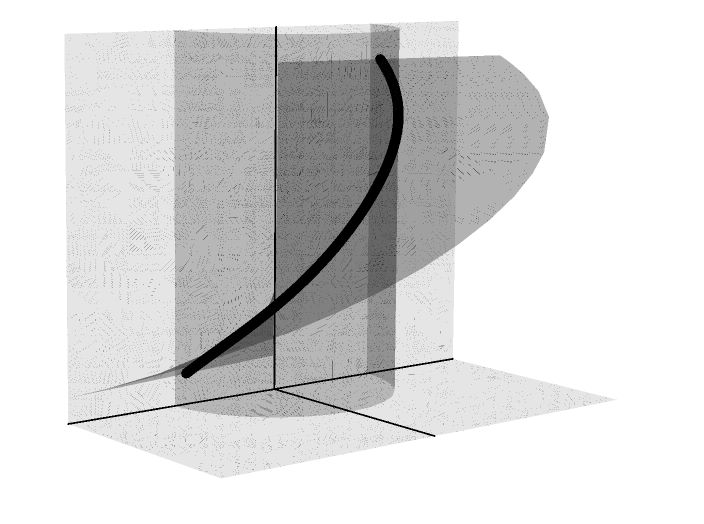}
 
 \caption{Minimal surfaces in $\h^2\times\r$ invariant by a one-parameter of translations. Left: vertical surface. Middle: parabolic surface. Right: hyperbolic surface}\label{figminimal}
 \end{figure}
 
%After Thm. \ref{t0}, in the next sections we will classify the grim reapers discarding that of this theorem. 

Although this paper is focused on grim reapers, it is natural, once two new notions of translators have been defined in $\h^2\times\r$, to ask about rotational examples. Here by rotational we mean a surface of revolution about a vertical geodesic. In the case of v-translators, the classification of rotational v-translators was obtained in \cite{bue1}. If we now consider the other two vector fields, that is, p-translators and h-translators, it is natural to expect that no many examples can exist because the rotational axis is orthogonal to the Killing vector fields. Indeed, we prove this in the following result.

\begin{theorem}\label{t02}
    Slices $\h^2\times\{z_0\}$, $z_0\in\r$, are the only p-translators and h-translators of rotational type.
\end{theorem}
\begin{proof}

Let $\Sigma$ be a rotational surface of $\h^2\times\r$. Without loss of generality, we assume that the rotation axis is the vertical geodesic  through the point $(0,1,0)$. We need a parametrization of $\Sigma$ which is a bit cumbersome in the upper half-space model: see \cite{on}. This parametrization can be found in the more convenient Lorentzian model for $\h^2\times\r$ and carry it to $\r^2_{+}\times\r$. The generating curve of $\Sigma$ is a curve $\alpha(s)=(y(s),z(s))$, $s\in I\subset\r$, where $y(s)\in (0,1)$ for all $s\in I$. The parametrization of $\Sigma$  is 
$$\Psi(s,t)= \left(-\frac{2 \left(y(s)^2-1\right) \sin t}{2 \left(y(s)^2+1-\left(y(s)^2-1\right) \cos t\right)},\frac{2 y(s)}{y(s)^2+1-\left(y(s)^2-1\right) \cos t},z(s)\right).$$
The value of $H$ is independent on the variable $t$. On the other hand, we compute the right hand-side of \eqref{eq22} and \eqref{eq23}. 
A first case to consider is that $y(s)$ is a constant function, $y(s)=y_0$. This means that $\Sigma$ is a vertical cylinder over a circle of $\h^2$.  Letting $z(s)\to s$ and after some computations, we obtain 
$$\langle N,\partial_x\rangle=\frac{\sin t}{y},\quad \langle N,x\partial_x+y\partial_y\rangle=y\cos t.$$
Since the left hand-side of \eqref{eq22} and \eqref{eq23} is independent on the variable $t$, we arrive to a contradiction and this case is not possible. 

Thus $y(s)$ is not a constant function. This means that we can consider $\alpha$ parametrized by $\alpha(s)=(s,z(s))$. By repeating the computations, we have 
$$\langle N,\partial_x\rangle=-\frac{s \sin t z'(s)}{\sqrt{s^2 z'(s)^2+1}},\quad \langle N,x\partial_x+y\partial_y\rangle=\frac{s \cos t z'(s)}{\sqrt{s^2 z'(s)^2+1}}.$$
In both cases, we conclude that $H=0$ and $z'(s)=0$  for all $s\in I$. Thus $z(s)$ is a constant function, $z(s)=z_0$. Then $\alpha$ is a horizontal line and $\Sigma$ is   (an open set  of) the slice $\h^2\times\{0\}$.  
\end{proof}

%%%%%%%%%%%%%%%%%%%%
\section{v-grim reapers}\label{sec4}
%%%%%%%%%%%%%%%%%%%%%%%%%%%%%%%%%%

In this section, we classify  the v-grim reapers by distinguishing if they are parabolic or hyperbolic.

%%%%%%%%%%%%%%%%%%%%

\subsection{Parabolic v-grim reapers}
%%%%%%%%
%%%%%%%%%%%%%%%%%%%%

The classification of  parabolic v-grim reapers has as result the explicit parametrization of the surface.

\begin{theorem}\label{t1} The parabolic v-grim reapers of $\h^2\times\r$ are parametrized by \eqref{p2} where
\begin{enumerate}
\item
\begin{equation}\label{gs1}
\alpha(s)=(y(s),z(s))=(e^{\frac{s}{\sqrt{5}}}, \frac{-2s}{\sqrt{5}}).
\end{equation}
The generating curve $\alpha(s)$ is an embedded vertical graph whose height function is unbounded and has a contact point with the ideal boundary $\h^2_\infty\times\r$ as $s\rightarrow-\infty$. See also \cite{on}.
\item 
\begin{equation}\label{gs}
\begin{split}
y(s)=&\left(\frac{10}{\sqrt{5} \sinh(\sqrt{5} s)+5 \cosh(\sqrt{5} s)}\right)^{1/5}\exp\left\{\frac{4}{5}\tan^{-1}(\frac{1}{2}(1+\sqrt{5}\tanh\frac{\sqrt{5}s}{2}))\right\},\\
z(s)=&\frac{2}{5}\left( \tan^{-1} (\frac{1}{2} (1+\sqrt{5}\tanh\frac{\sqrt{5}s}{2} ))+\log(\frac{5\cosh(\sqrt{5}s)+\sqrt{5}\sinh(\sqrt{5}s)}{10})\right).
\end{split}
\end{equation}
The generating curve $\alpha(s)$ is an embedded vertical bi-graph whose height function $z(s)$ attains a global minimum and $y(s)$ a global maximum. Moreover, the surface has two contact points with   the ideal boundary $\h^2_\infty\times\r$ as $|s|\rightarrow\infty$.
\end{enumerate}
\end{theorem}

\begin{proof} Suppose that $\Sigma$ is a parabolic v-grim reaper parametrized by \eqref{p2}, where the generating curve $\alpha$ satisfies the equations \eqref{sp}.  From \eqref{normal-p}, we know $\langle N,\partial_z\rangle=-\cos\theta$. Using \eqref{mean-p}, Equation  \eqref{eq21} we deduce
\begin{equation}\label{eqp}
\theta'=2\cos\theta+\sin\theta.
\end{equation}
A trivial solution of this equation is $\theta=-\arctan 2$, which yields after integration \eqref{gs1}. If $\theta$ is not constant then
$$
\theta(s)=2\tan^{-1}(\frac{1}{2}(1+\sqrt{5}\tanh\frac{\sqrt{5}s}{2})).
$$
From this expression of $\theta$, we can solve \eqref{sp}, obtaining as solutions \eqref{gs}. The properties of the surface are immediate by the explicit parametrization of $\alpha$.  See Fig. \ref{figPGR}, left and middle.
\end{proof}

\begin{figure}[hbtp]
\begin{center}
\includegraphics[width=.15\textwidth]{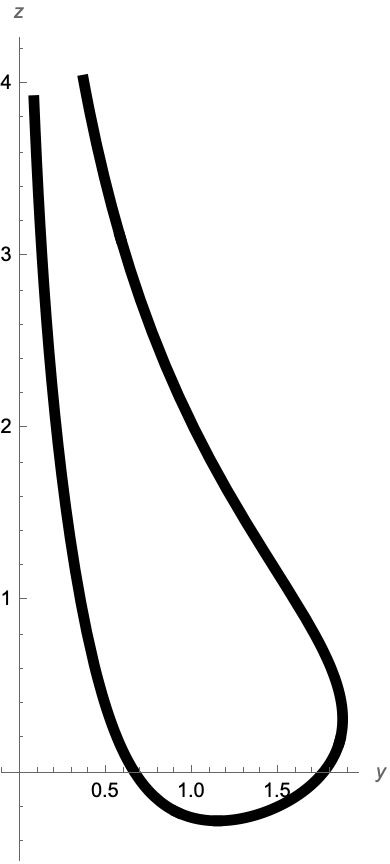}
\includegraphics[width=.15\textwidth]{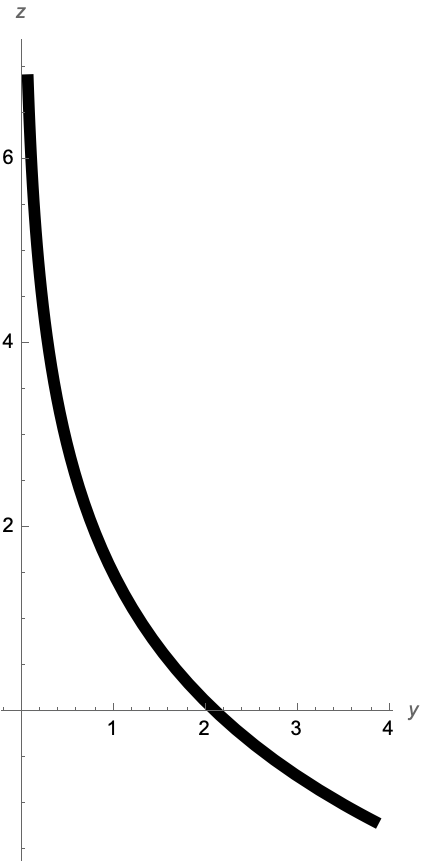}
\includegraphics[width=.3\textwidth]{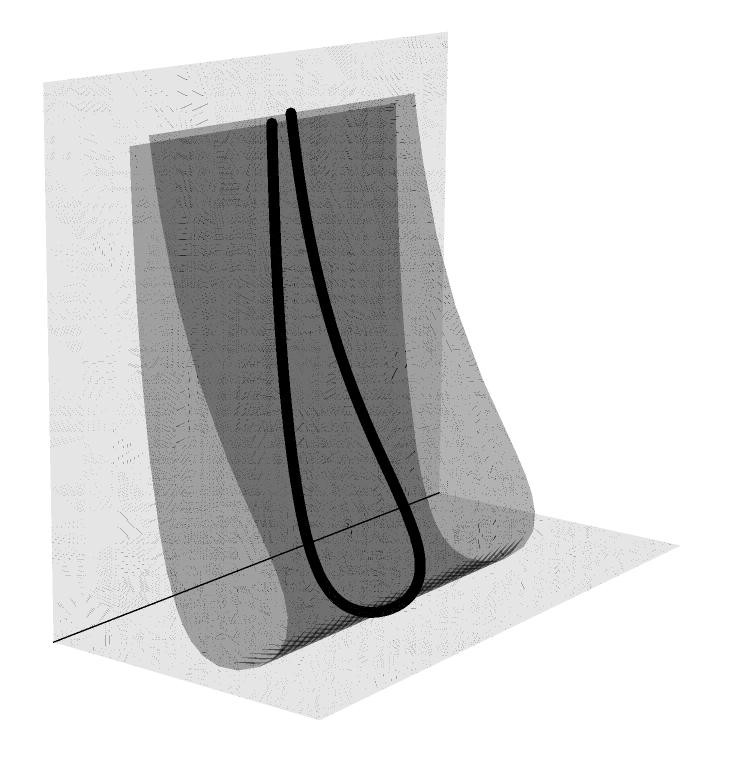}
\includegraphics[width=.3\textwidth]{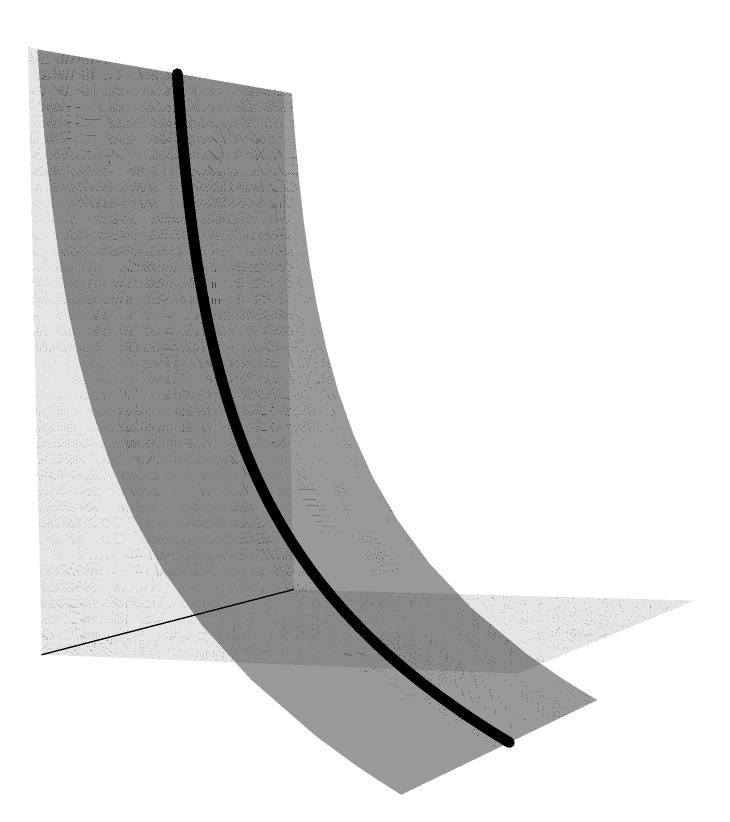}
\end{center}
\caption{The first two pictures show the generating curves of   parabolic v-grim reapers. The last two show the corresponding parabolic v-grim reapers.}
\label{figPGR}
\end{figure}

%%%%%%

\begin{remark} Consider v-translators in the space $\h^n\times\r$ with coordinates $(\textbf{x},z),\ \textbf{x}\in\h^n,\ z\in\r$.  In Thm. 11 of \cite{lima},  the authors prove the existence of v-translators foliated by horospheres at each hyperplane $z=c,\ c\in\r$, (see also \cite{lipi}). These hypersurfaces are parabolic v-translators in our notation. In  Thm. \ref{t1} (case $n=2$), we have obtained explicit parametrizations of the surfaces and their main geometric properties. 
\end{remark}

As in the Euclidean case $\r^3$, we can ask about the existence of tilted v-grim reapers in the context of $\h^2\times\r$. Recall that in $\r^3$, the translators invariant by a one-parameter group of translations are planes containing the vector $\textbf{v}$ in \eqref{eq1} and grim reapers, non-planar translators parametrized as ruled surfaces. There are two types of grim reapers. First, those ones whose rulings are orthogonal to $\textbf{v}$ and where the generating curve is a grim reaper curve. A second type of grim reapers are when the rulings are not orthogonal to $\textbf{v}$. The generating curve is not a grim reaper curve but a suitable dilation and translation in the parameters of that curve. 

Coming back to $\h^2\times\r$, and searching for tilted grim reapers,   it is natural to consider v-translators that are ruled surfaces whose rulings are parallel to a fixed vector $\textbf{v}=(v_1,0,v_3)$ in the $xz$-plane, where $v_1,v_3\not=0$. Notice that the vector field $v_1\partial_x+v_3\partial_z$ is a Killing vector field but the flow of the infinitesimal isometries that generate are not translations of $\h^2\times\r$. Anyway, we can ask for those v-translators parametrized  by 
\begin{equation}\label{p12}
\Psi(s,t)=(y(s),z(s))+t(v_1,0,v_3),\quad s\in I,t\in\r.
\end{equation}
These translators will be called tilted v-grim reapers. After a reparametrization, we can suppose in \eqref{p12} that $v_1=1$ with $v_3\not=0$.

\begin{figure}[hbtp]
\begin{center}
\includegraphics[width=.12\textwidth]{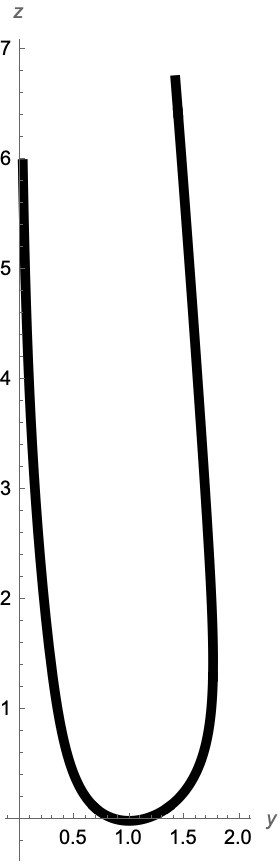}\hspace{1.5cm}
\includegraphics[width=.22\textwidth]{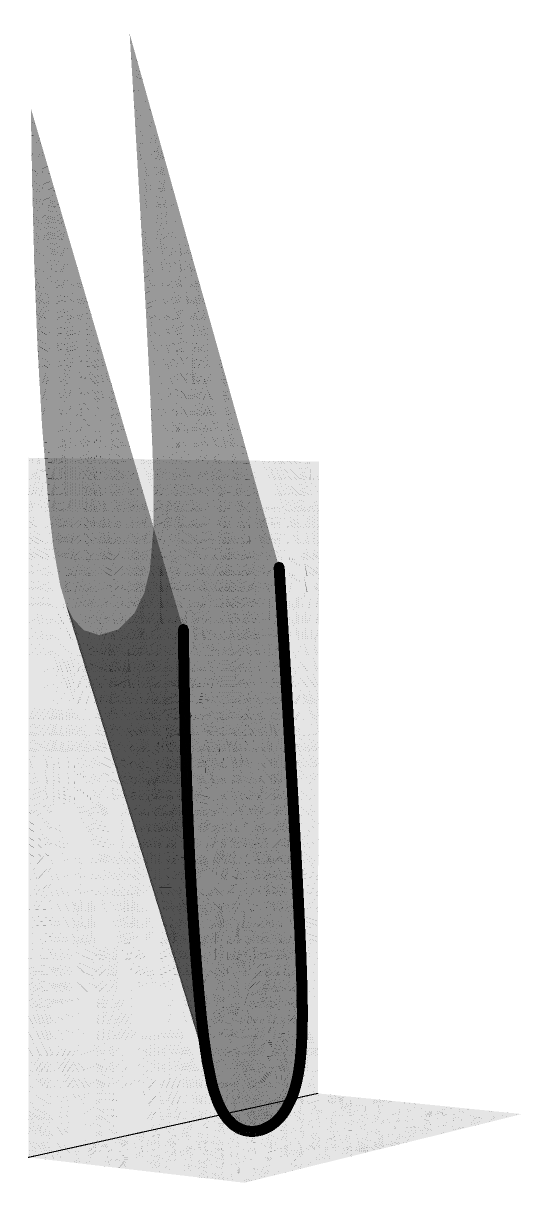}
\end{center}
\caption{Left: the generating curve of a tilted parabolic v-grim reaper. Right: a tilted v-grim reaper.}
\label{figPGRtilted}
\end{figure}

\begin{theorem}\label{thtilted}
Let $\Sigma$ be a tilted v-grim reaper parametrized by \eqref{p12}. Suppose that the generating curve $\alpha(s)=(y(s),z(s))$ is parametrized by \eqref{sp}. Then the function $\theta$ satisfies the ODE
\begin{equation}\label{eqp2}
\theta'= 2\cos\theta+\sin\theta\frac{1+2v_3^2\cos^2\theta y^2}{1+v_3^2y^2} .
\end{equation}
Moreover, $\Sigma$ is an embedded vertical bi-graph and has two contact points with the ideal boundary $\h^2_\infty\times\r$ as $|s|\rightarrow\infty$. 
\end{theorem} 
Solutions of \eqref{eqp2} cannot be obtained explicitly as in the case of \eqref{eqp}. In Fig. \ref{figPGRtilted}, left, we show the generating curve of a tilted v-grim reaper, and in Fig. \ref{figPGRtilted}, right, we see the surface of a tilted v-grim reaper. Analogously to the Euclidean case, the generating curves are similar as in the case that $\textbf{v}=(1,0,0)$ and whose parametrizations are given in \eqref{gs}. Again, the surface is an embedded  bi-graph with two points at   $\h^2_\infty\times\r$.

%\begin{figure}[hbtp]
%\begin{center}
%\includegraphics[width=.1\textwidth]{parabolicGR21.png}\hspace*{3cm} 
%\includegraphics[width=.2\textwidth]{parabolicGR22.png}
%\end{center}
%\caption{Tilted v-grim reapers. Here $v_3=1$. Left: the generating curve of the tilted v-grim reaper.   Right: the tilted v-grim reaper.}
%\label{figPGR2}
%\end{figure}

\subsection{Hyperbolic v-grim reapers}
%%%%%%%%

We now study hyperbolic v-grim reapers. In contrast to the parabolic v-grim reapers, we cannot give explicit parametrizations of the surfaces. To describe the geometric properties of these surfaces, we will study the phase plane associated to the generating curves. Part of this section  coincides with \cite{lipi} in their study of   the $1$-mean curvature flow (see also \cite[Thm. 12]{lima}). 
 
A first example of a hyperbolic v-grim reaper is   the vertical plane of equation $x=0$. Indeed, this surface is invariant by the group $\mathcal{H}$ and because its generating curve is the geodesic $x=z=0$, the surface is minimal. Since $N$ is horizontal, then $\langle N,\partial_z\rangle=0$. Thus the surface satisfies \eqref{eq21}. 

For the general classification of the hyperbolic v-grim reapers, we know from Sect. \ref{sec2} that the generating curve of a hyperbolic v-grim reaper is $\alpha(s)=(\cos r(s),\sin r(s),z(s))$, where $r$ and $z$ satisfy \eqref{sh}. From \eqref{normal-h}, we have $\langle N,\partial_z\rangle=-\cos\rho$, and by the expression of $H$ in \eqref{mean-h}, Equation  \eqref{eq21} is 
$$
\rho'=2\cos\rho+\cos r\sin\rho.
$$
We study the properties of this ODE system by projecting the solutions into the autonomous   second order system
\begin{equation}\label{ODEsystem}
\left\{
\begin{array}{l}
r'=\sin r\cos\rho\\
\rho'=2\cos\rho+\sin\rho\cos r,
\end{array}
\right.\qquad r\in(0,\pi),\ \rho\in(-\pi,\pi).
\end{equation}
The phase plane of \eqref{ODEsystem} is the set
$$
\Theta=\{(r,\rho)\colon r\in(0,\pi),\ \rho\in(-\pi,\pi)\},
$$
with coordinates $(r,\rho)$. The two equilibrium points are $(\pi/2,\pi/2)$ and $(\pi/2,-\pi/2)$. These points correspond to the case that the generating curve $\alpha$ is the vertical straight line $z\mapsto (0,1,z),\ z\in\r$, parametrized with increasing height if $\rho=\pi/2$ and with decreasing height if $\rho=-\pi/2$. Therefore, the surface is the vertical plane of equation $x=0$. This solution is already known.

  The orbits are the solutions $\gamma(s)=(r(s),\rho(s))$ of \eqref{ODEsystem} when regarded in $\Theta$, and they foliate $\Theta$ as a consequence of the existence and uniqueness of the Cauchy problem of \eqref{ODEsystem} for initial conditions $(r_0,\rho_0)\in\Theta$.

In virtue of the equations of \eqref{ODEsystem}, for any orbit $\gamma(s)=(r(s),\rho(s))$ we have $r'(s)>0$, being only zero at the boundary lines $\rho=\pm\pi/2$. For the coordinate $\rho$, we first define the curve $\Lambda=\{\rho=\Lambda(r)\}\cap\Theta$, where
$$
\Lambda(r)=-\arctan\frac{2}{\cos r},\qquad r\neq\pi/2,
$$
which is a non-connected graph on the $r$-axis. For $r>\pi/2$, one end of $\Lambda$ converges to the equilibrium $(\pi/2,\pi/2)$ when $r\searrow\pi/2$ and the other converges to $(\pi,\arctan 2)$. For $r<\pi/2$, one end converges to the equilibrium $(\pi/2,-\pi/2)$ when $r\nearrow\pi/2$ and the other converges to $(0,-\arctan 2)$. If $r(s)>\pi/2$ then $\rho'(s)>0$ if and only if $\rho(s)<\Lambda(r(s))$. Analogously, if $r(s)<\pi/2$ then $\rho'(s)>0$ if and only if $\rho(s)>\Lambda(r(s))$.

Next, we see that we can restrict $\Theta$ to $r\in(0,\pi/2],\ \rho\in(-\pi/2,\pi/2)$. First, note that if $\gamma(s)=(r(s),\rho(s))$ is an orbit then $\widetilde{\gamma}(s)=(r(-s),\rho(-s)\pm\pi)$ is again an orbit. Consequently, if $\gamma$ lies in the strip $(\pi/2,\pi)$, (resp. $(-\pi,-\pi/2)$), then $\widetilde{\gamma}(s)=(r(-s),\rho(-s)-\pi)$, (resp. $\widetilde{\gamma}(s)=(r(-s),\rho(-s)+\pi)$) lies in the strip $(-\pi/2,\pi/2)$. To reduce the variable $r\in(0,\pi/2]$ just bear in mind that the orbits of \eqref{ODEsystem} are anti-symmetric with respect to the axes $r=\pi/2,\rho=0$. This comes from the fact that if $\gamma(s)=(r(s),\rho(s))$ is an orbit, then $\widetilde{\gamma}(s)=(\pi-r(-s),-\rho(-s))$ is again an orbit.

In the following result, we classify the solutions \eqref{ODEsystem}, given a complete geometric description of the generating curves. By vertical graph we will mean a graph on the $xy$-plane. 
\begin{theorem}\label{thorbits}
Let be $r_0\in(0,\pi/2]$ and $\gamma_{r_0}(s)=(r(s),\rho(s))$ the orbit passing through $(r_0,0)$ at $s=0$. Let $\alpha_{r_0}$ be the curve corresponding to $\gamma_{r_0}$ with initial condition $z(0)=0$.
\begin{enumerate}
\item If $r_0=\pi/2$, then $\gamma_{\pi/2}$ is anti-symmetric about the axes $r=\pi/2,\rho=0$. When $s>0$ increases then $\gamma_{\pi/2}$ intersects $\Lambda$ and $\gamma_{\pi/2}(s)\rightarrow(\pi,\arctan 2)$ as $s\rightarrow\infty$. When $s<0$ decreases then $\gamma_{\pi/2}(s)$ intersects $\Lambda$ and $\gamma_{\pi/2}(s)\rightarrow(0,-\arctan 2)$ as $s\rightarrow-\infty$. 

The curve $\alpha_{\pi/2}$ is a vertical graph and symmetric about the vertical plane $x=0$.
\item There exists $r_*<\pi/2$ such that $\gamma_{r_*}(s)\rightarrow(\pi/2,\pi/2)$ as $s\rightarrow\infty$, while $\gamma_{r_*}(s)$ intersects $\Lambda$ and converges to $(0,-\arctan 2)$ as $s\rightarrow-\infty$.

The curve $\alpha_{r_*}$ is a vertical graph converging to the vertical  line at the point $(0,1,0)$ as $s\rightarrow\infty$, with $r(s)\rightarrow\pi/2$ and $z(s)\rightarrow\infty$.
\item If $r_0\in(r_*,\pi/2)$, when $s>0$ increases then $\gamma_{r_0}(s)$ intersects $\Lambda$ and converges to $(\pi,\arctan 2)$ as $s\rightarrow\infty$. When $s<0$ decreases then $\gamma_{r_0}(s)$ intersects $\Lambda$ and converges to $(0,-\arctan 2)$ as $s\rightarrow-\infty$.

The curve $\alpha_{r_0}$ is a vertical graph.
\item If $r_0\in(0,r_*)$, when $s>0$ increases then $\gamma_{r_0}$ lies at the left-hand side of $\gamma_{r_*}$, intersects the line $\rho=\pi/2$ where its $r$-coordinate attains a maximum and then $\gamma_{r_0}(s)\rightarrow(0,\pi-\arctan 2)$ as $s\rightarrow\infty$. When $s<0$ decreases then $\gamma_{r_0}$ intersects $\Lambda$ and converges to $(0,-\arctan 2)$ as $s\rightarrow-\infty$.

The curve $\alpha_{r_0}$ fails to be a vertical graph just at the time $s_0>0$ for which $\rho(s_0)=\pi/2$.
\end{enumerate}
In all the cases, $\alpha_{r_0}$ is strictly contained in $\r^2_+\times[0,\infty)$, being tangent at $z=0$ exactly at the point $\alpha_{r_0}(0)$.
\end{theorem}

\begin{proof}
Let $r_0\in(0,\pi/2]$, fix the initial condition $(r_0,0)$ and let $\gamma_{r_0}$ be the solution of Eq. \eqref{ODEsystem} such that $\gamma_{r_0}(0)=(r_0,0)$.

First, assume $r_0=\pi/2$. Then, $\gamma_{\pi/2}$ is anti-symmetric about the axes $r=\pi/2,\rho=0$. For $s>0$ increasing, both of its coordinates $r,\rho$ increase until $\gamma_{\pi/2}$ intersects $\Gamma$ at some point $(r_{\pi/2},-\arctan\frac{2}{\cos r_{\pi/2}})$. Then, $r$ keeps increasing and $\rho$ decreases when $\gamma_{\pi/2}(s)\rightarrow(\pi,\arctan 2)$ as $s\rightarrow\infty$. By the anti-symmetry of $\gamma_{\pi/2}$, we conclude that $\gamma_{\pi/2}(s)\rightarrow(0,-\arctan 2)$ as $s\rightarrow-\infty$. See Fig. \ref{figHGR}, left, the orbit in blue. The curve $\alpha_{\pi/2}$ associated to $\gamma_{\pi/2}$ is a vertical graph, since $z'=\sin\rho$ in virtue of \eqref{sh} and $|\rho|<\pi/2$ and symmetric about the vertical plane $x=0$, as well as the corresponding v-grim reaper. See Fig. \ref{figHGR}, right, the curve in blue. This proves (1). 

\begin{figure}[hbtp]
\begin{center}
\includegraphics[width=.4\textwidth]{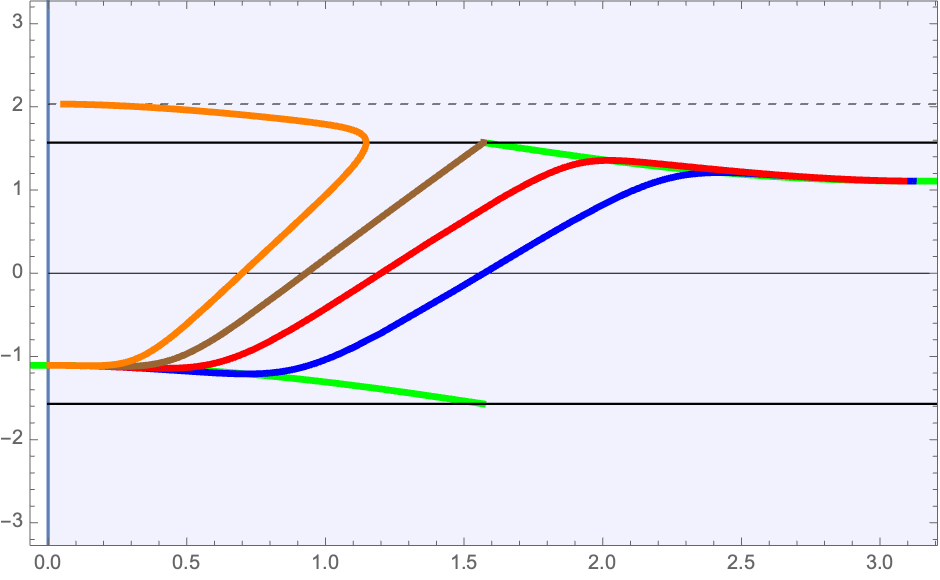}
\includegraphics[width=.4\textwidth]{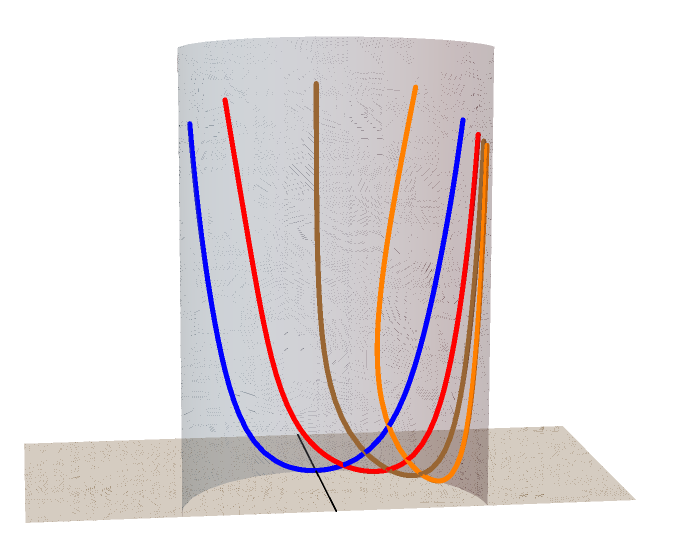}
\end{center}
\caption{Left: the behavior of the different solutions of \eqref{ODEsystem}. Right: the corresponding generating curves of the v-grim reapers.}
\label{figHGR}
\end{figure}

Let $r_0<\pi/2$ be close enough to $\pi/2$. Then, $\gamma_{r_0}$ has to intersect $\Gamma$ as $s$ increases and then converge to $(\pi,\arctan 2)$ by monotonicity. Recall that when $\gamma_{r_0}$ intersects $\Gamma$, it defines a unique point $(\widetilde{r}_0,-\arctan\frac{2}{\widetilde{r}_0})$. Moreover, if $r_0<r_1$ then $\widetilde{r}_0<\widetilde{r}_1$ since $\gamma_{r_0}$ and $\gamma_{r_1}$ cannot intersect. Let us restrict $\Gamma$ to the values $r\in(\pi/2,r_{\pi/2}]$. By existence and uniqueness, for each $(\widetilde{r}_0,-\arctan\frac{2}{\cos\widetilde{r}_0})\in\Gamma$, there is a unique $r_0\in(0,\pi/2]$ such that the orbit $\gamma_{r_0}$ intersects $\Gamma$ at $(\widetilde{r}_0,-\arctan\frac{2}{\cos\widetilde{r}_0})$. See Fig. \ref{figHGR}, left, the orbit in red. Again,$|\rho|<\pi/2$ hence $\alpha_{r_0}$ is a vertical graph.

When $\widetilde{r}_0\rightarrow\pi/2$, then $r_0\rightarrow r_*$ for some $r_*\geq0$. We assert that $r_*$ is positive. Indeed, an orbit $\gamma$ passing through some $(\widetilde{r}_0,\pi/2)$ must end up intersecting the line $\rho=0$ as the parameter $s$ decreases. See Fig. \ref{figHGR}, left, the orbit in orange. The curves associated to these orbits fail to be vertical graphs, since $\rho(s_0)=\pi/2$ for some $s_0$ and hence $\alpha_{r_0}'(s_0)$ is vertical. This proves that $r_*=0$ cannot happen, since otherwise this orbit and $\gamma_0$ would intersect, a contradiction.

Note that for every $(\widetilde{r}_0,\pi/2)$ we obtain a point $(r_0,0),\ r_0>0$. Moreover, if $\widetilde{r}_0<\widetilde{r}_1$ then $r_0<r_1$. Again, as $\widetilde{r}_0\rightarrow\pi/2$ we have $r_0\rightarrow\hat{r}$. At this point, it may happen $\hat{r}<r_*$ and for every $r_0\in[\hat{r},r_*]$, the orbit $\gamma_{r_0}\rightarrow(\pi/2,\pi/2)$. We prove that this case cannot occur and therefore $\hat{r}=r_*$. 

As usual, let be $\gamma_{\hat{r}}=(\hat{r}(t),\hat{\rho}(t))$ and $\gamma_{r_*}=(r_*(s),\rho_*(s))$ the orbits that pass through $(\hat{r},0)$ and $(r_*,0)$, respectively. We use the parameter $t$ to refer to $\gamma_{\hat{r}}$ and $s$ to refer to $\gamma_{r_*}$. First, see that
$$
\gamma_{\hat{r}}'(0)=(\sin\hat{r},2),\qquad \gamma_{r_*}'(0)=(\sin r_*,2).
$$
Moreover, fix some $\rho_0\in(0,\pi/2)$ and let $t_0,s_0$ such that $\gamma_{\hat{r}}(t_0)$ and $\gamma_{r_*}(s_0)$ both intersect the line $\rho=\rho_0$, that is $\hat{\rho}(t_0)=\rho_*(s_0)$.
Since $\hat{r}(t_0)<r_*(s_0)$, it is immediate that 
$$
\frac{\hat{\rho}'(t_0)}{\hat{r}'(t_0)}>\frac{\rho_*'(s_0)}{r_*'(s_0)}.
$$
Geometrically, if we express the orbits $\gamma_{\hat{r}}$ and $\gamma_{r_*}$ as graphs $\hat{\rho}=\hat{\rho}(r)$ and $\rho_*=\rho_*(r)$, then $\hat{\rho}'(\hat{r})>\rho_*'(r_*)$. Consequently, it cannot happen that both $\gamma_{\hat{r}}$ and $\gamma_{r_*}$ converge to $(\pi/2,\pi/2)$ unless $\hat{r}=r_*$. This proves the existence and uniqueness of an orbit $\gamma_{r_*}$ converging directly to $(\pi/2,\pi/2)$. See Fig. \ref{figHGR}, left, the orbit in brown. The curve $\alpha_{r_*}$ is a vertical graph since $|\rho|<\pi/2$, but $\rho(s)\rightarrow\pi/2$ as $s\rightarrow\infty$, hence $\alpha_{r_*}$ tends to be vertical. Since $r(s)\rightarrow\pi/2$ as $s\to\infty$ we conclude that $\alpha_{r_*}$ converges to the vertical line $z\mapsto (0,1,z)$. This proves the item (2). 

In all the cases, the $z$-coordinate of the generating curve $\alpha_{r_0}$ satisfies $z'(s)=\sin\rho(s)$, hence $z'$ only vanishes at the instant $s=0$ where $z(s)$ attains a minimum. Consequently, $z(s)\geq0$, which yields that $\alpha_{r_0}$ is contained in $\r^2_+\times[0,\infty)$ and $\alpha_{r_0}$ is tangent at $z=0$ exactly at $\alpha_{r_0}(0)$. This concludes the classification of the generating curves of the hyperbolic v-grim reapers.
\end{proof}

The v-grim reapers generated by moving each of their generating curves along hyperbolic translations share their same properties. In Fig. \ref{figHGR2} we plot the hyperbolic v-grim reaper corresponding to  Thm. \ref{thorbits}.

\begin{figure}[hbtp]
\begin{center}
\hspace{-1cm}
\includegraphics[width=.32\textwidth]{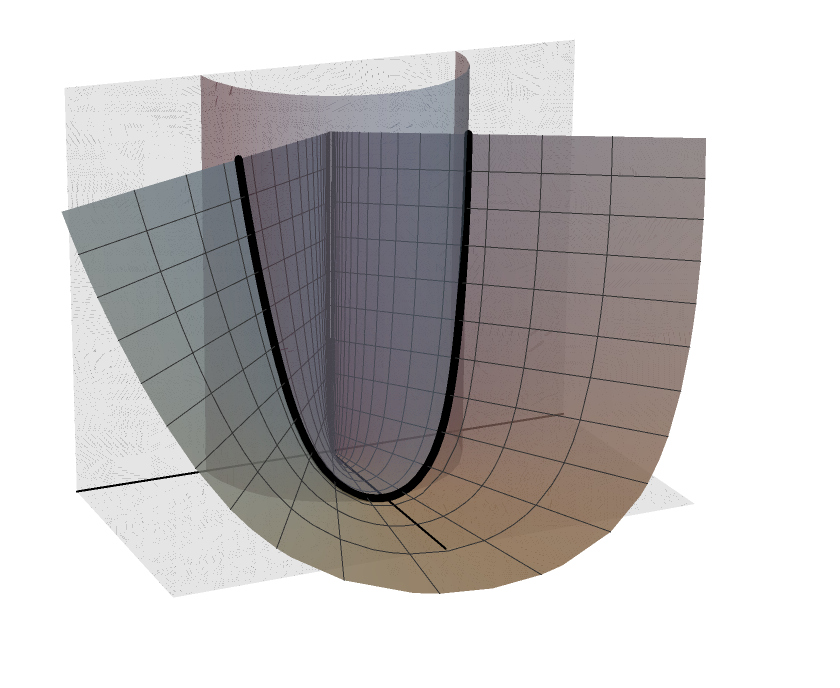}
\includegraphics[width=.35\textwidth]{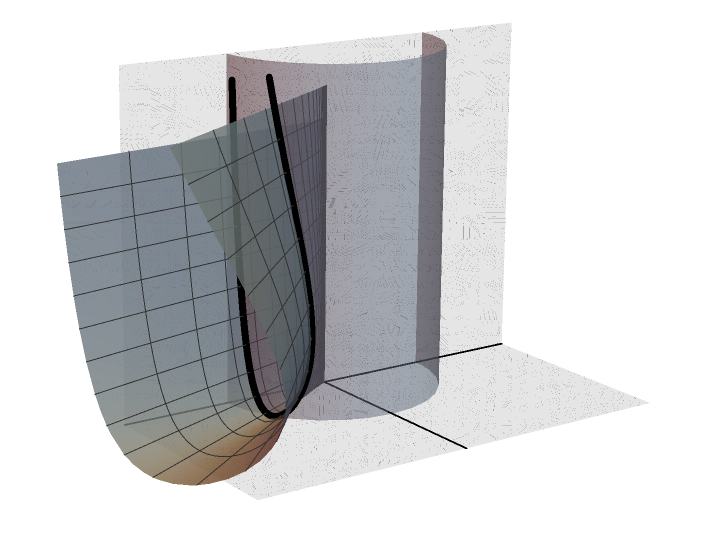}
\includegraphics[width=.35\textwidth]{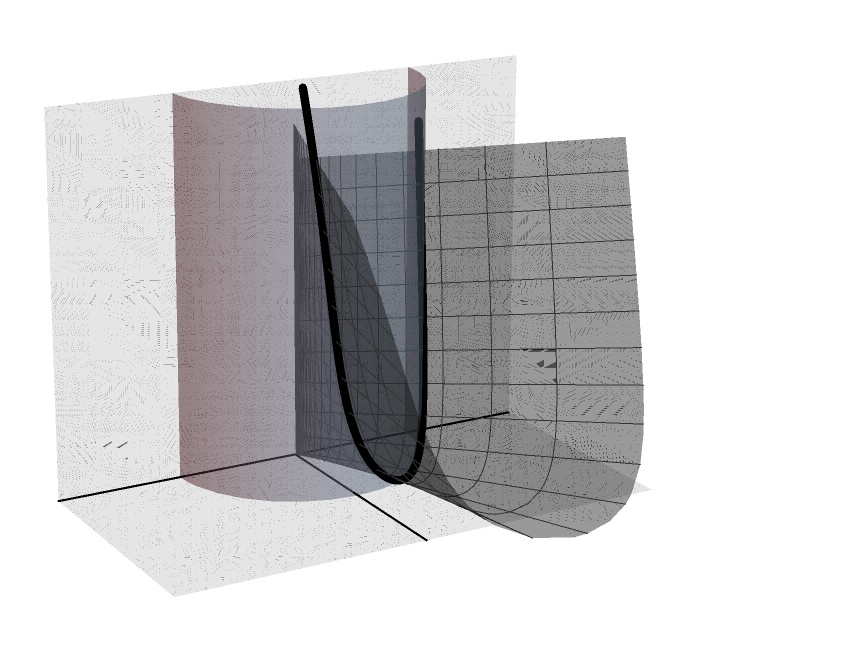}
\end{center}
\caption{Examples of hyperbolic v-grim reapers.}
\label{figHGR2}
\end{figure}

%%%%%%%%%%
\section{p-grim reapers}\label{sec5}
%%%%%%%%%%%%%%%%%%%

In this section we classify p-grim reapers. When  the surface is invariant by parabolic translations, we know that the surface is minimal (Thm. \ref{t0}). This case will be discarded. We begin with vertical p-grim reapers.

\begin{theorem}
Let $\Sigma$ be a vertical p-grim reaper parametrized by \eqref{p1}, where the generating curve $\alpha(s)=(x(s),y(s))$ satisfies \eqref{sv}. Then  
\begin{equation}\label{511}
\theta'=-\cos\theta-\frac{2\sin\theta}{y}.
\end{equation}
Moreover, $\alpha$ is a bi-graph over $y=0$ and converges to it as $x\rightarrow\infty$.
\end{theorem}

\begin{proof}
We know that  $x'(s)=y(s)\cos\theta(s),\ y'(s)=y(s)\sin\theta(s)$. Then, by the expression of $N$ and $H$ in \eqref{normal-v} and \eqref{mean-v}, the equation \eqref{eq22} writes as  \eqref{511}. Then $\theta$ and the coordinate functions of $\alpha$ are solutions of the nonlinear autonomous system
$$
\left\{
\begin{array}{l}
x'=y\cos\theta,\\
y'=y\sin\theta,\\
\theta'=-\cos\theta-\dfrac{2\sin\theta}{y}.
\end{array}
\right.
$$
Since the second equation can be obtained by solving the former and the latter, we project any solution   to the $(y,\theta)$-plane and study the system
$$
\left\{
\begin{array}{l}
y'=y\sin\theta,\\
\theta'=-\cos\theta-\dfrac{2\sin\theta}{y}.
\end{array}
\right.
$$

\begin{figure}[hbtp]
\begin{center}
\includegraphics[width=.4\textwidth]{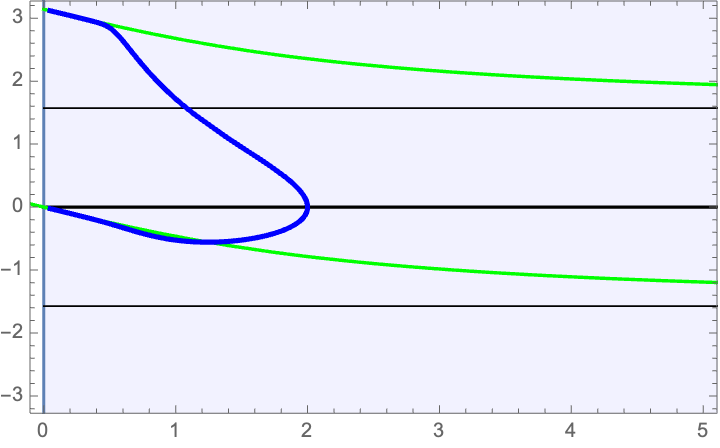}\qquad
\includegraphics[width=.3\textwidth]{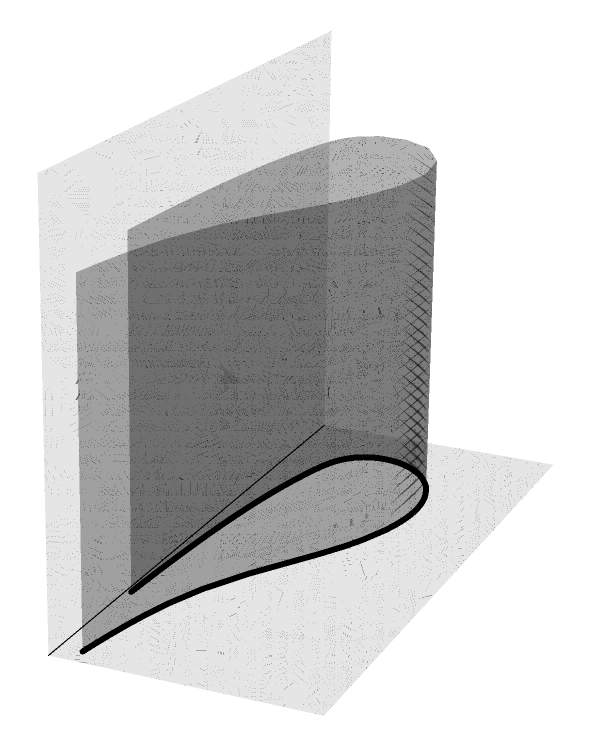}
\end{center}
\caption{Left: the solutions of the nonlinear autonomous system fulfilled by the generating curve of a vertical p-grim reaper. Right: a vertical p-grim reaper  }\label{fig5}
\end{figure}

At this point, we remark that a similar system appeared in \cite{bl1} in the framework of translating solitons in $\h^3$ with respect to the Killing vector field $\partial_x$. The only difference here with the work \cite{bl1} is that   we now consider $\alpha(s)=(x(s),y(s))$ parametrized by the hyperbolic arc-length, and in \cite{bl1} the curve was parametrized by the Euclidean arc-length $x'^2+y'^2=1$. Nevertheless, the discussion for this system is similar to the one depicted in \cite{bl1}. We only summarize the main ideas in order to make this proof self-contained.

The function $\theta$ can be restricted to lie in $\theta\in(-\pi/2,\pi)$; every orbit is defined by the initial condition $(y_0,0),\ y_0>0$, where $y_0$ stands for the maximum (Euclidean) distance of the orbit to $y=0$; and every orbit is a compact arc converging to the points $(0,0)$ and $(\pi,0)$ as $s\rightarrow\pm\infty$ respectively. See Fig. \ref{fig5}, left. The generating curve $\alpha$ fails to be a graph over $y=0$ precisely at the instant when the orbit intersects the line $\theta=\pi/2$, where $x'=0$. Finally, in \cite{bl1} it was proved that when the parameter $s$ of any orbit diverges to $\infty$ (resp. $-\infty$), the orbit converges to $(0,0)$ (resp. $(0,\pi)$).

Consequently, $\alpha$ is a bi-graph over $y=0$. Since the parameter $s\rightarrow\pm\infty$ and $\theta\rightarrow0,\pi$, the curve $\alpha$ cannot converge to $y=0$ with a cusp point, hence its $x$-coordinate satisfies $x(s)\rightarrow\infty$ as $s\rightarrow\pm\infty$. See Fig. \ref{fig5}, right.

\end{proof}

We finish this section with hyperbolic p-grim reapers. We find that surfaces are not only minimal, but actually only horizontal slices.

\begin{theorem}
Slices $\h^2\times\{z_0\}$, $z_0\in\r$, are the only hyperbolic p-grim reapers.
\end{theorem}

\begin{proof} Suppose that $\Sigma$ is parametrized by \eqref{p2}, where the generating curve $\alpha$ satisfies \eqref{sh}. Using \eqref{normal-h}, we have $\langle N,\partial_x\rangle=e^{-t}\sin\rho$. From \eqref{mean-h}, Equation \eqref{eq22} is 
$$\rho'-\cos r\sin\rho=-2 e^{-t}\sin\rho.$$
Since this identity holds for all $t\in\r$, we deduce  $H=0$ and $\sin\rho=0$. From $\rho=0$, we have $z(s)$ is constant, $z(s)=z_0\in\r$. The solution of $H=0$ gives the half-circles
$$
x(s)=\cos(2\arctan e^s),\qquad y(s)=\sin(2\arctan e^s).
$$
Therefore, $\Sigma$ is the horizontal plane of equation $z=z_0$, that is, the slice $\h^2\times\{z_0\}$.  
\end{proof}

 %%%%%%%%%%%%%%%%%%%%%%%%%%%%%%%%%
\section{h-grim reapers}\label{sec6}
%%%%%%%%%%%%%%%%%%%

In this section we classify h-grim reapers. Recall that the case when  the surface is invariant by hyperbolic translations, then the surface is minimal (Theorem \ref{t0}).   In the following result, we approach the case when the h-grim reaper is vertical.
 
\begin{theorem} Let $\Sigma$ be a vertical h-grim reaper parametrized by \eqref{p1}, where the generating curve $\alpha(s)=(x(s),y(s))$ satisfies \eqref{sv}. Then $\theta$ satisfies 
\begin{equation}\label{611}
\theta'=\cos\theta-\frac{2x\sin\theta}{y}.
\end{equation}
\end{theorem}

\begin{proof} The fact that $\theta$ is a solution to the aforementioned ODE is immediate from \eqref{normal-v}, \eqref{mean-v} and because $\langle N,x\partial x+y\partial_y\rangle= x\sin\theta/y$. By \eqref{eq23}, the function $\theta$ satisfies \eqref{611}. Together \eqref{sh}, the curve $\alpha$ satisfies 
\begin{equation*}
\left\{
\begin{split}
x'&=y\cos\theta,\\
y'&=y\sin\theta,\\
\theta'&=\cos\theta-\frac{2x\sin\theta}{y}.
\end{split}
\right.
\end{equation*}
From this system we see that the generating curve $\alpha$ is symmetric with respect to the line $x=z=0$.

At this point, we cannot project a solution $(x,y,\theta)$ to an autonomous $2D$-system because of the dependence of both $x,y$ in the expression of $\theta'$. If we express $\alpha(s)$ as a graph $(x,y(x))$, the mean curvature and the product $\langle N,x\partial_x+y\partial_y\rangle$ are
$$
H=-\frac{1+y'^2+yy''}{2(1+y'^2)^{3/2}},\qquad \langle N,x\partial_x+y\partial_y\rangle=\frac{y'}{y\sqrt{1+y'^2}},
$$
hence $y(x)$ is a solution of
\begin{equation}\label{eqodeverticalh}
y''=\frac{(y-2xy')(1+y'^2)}{y^2}.
\end{equation}
By the even condition on $\alpha$, it suffices to consider $x\geq0$. Let us assume initial conditions $y(0)=y_0>0$, $y'(0)=0$. Then, $y''(0)=1/y_0>0$ and $y(x) $ attains at $x=0$ a minimum. Moreover, this is its unique critical point since $y'(x_0)=0$ implies $y''(x_0)=1/y(x_0)>0$ and hence would be another minimum, making impossible the existence of a maximum or an inflection point. In particular, the right-hand side of \eqref{eqodeverticalh} is bounded and hence $y(x)$ is defined for every $x\in\r$.

Now, for $x>0$ increasing the function $y(x)$ also increases. Thus  $y'(x)>0$ for $x>0$ since $y'(x)$ never vanishes again. We claim that $y''(x)$ vanishes at some point and hence $y(x)$ changes its curvature. Arguing by contradiction, assume that $y''(x)>0$ for every $x>0$. By \eqref{eqodeverticalh} this would imply
$$
y(x)>2xy'(x),\qquad\forall x>0.
$$
Now fix some $x_0>0$ and an integration from $x_0$ to $x>x_0$ yields
$$
\log x-\log x_0>2\log\frac{y(x)}{y(x_0)}.
$$
At this point, letting $x_0\rightarrow0$, we conclude that the left-hand side of this equation is eventually negative, a contradiction since its right-hand side is always positive. Consequently, $y(x)$ is strictly convex around $x=0$ and then changes its curvature at some $x_0>0$. See Fig. \ref{fig6}, left for an example of vertical h-grim reaper. 
\end{proof}

In the case of parabolic h-grim reapers we get an explicit parametrization of the surface. 

\begin{theorem} \label{t62}
The only parabolic  h-grim reapers are slices $\h^2\times\{z_0\}$,  $z_0\in\r$, and surfaces parametrized by \eqref{p2}, where the generating curve is 
\begin{equation}\label{vh}
\alpha(s)=(y(s),z(s))=2(c_1\cosh(s),  c_2+\tan^{-1} e^s),\quad c_1,c_2\in\r, c_1>0.
\end{equation}
Each surface is a bi-graph and it is contained between the slices $\h^2\times\{z_1\}$ and $\h^2\times\{z_2\}$, to which it is asymptotic. 
\end{theorem}

\begin{proof} Using \eqref{normal-p}, we have $\langle N,\partial_x\rangle= \sin\theta$. From \eqref{mean-p}, Equation \eqref{eq22} is 
$$
\theta'=-\sin\theta.
$$
A trivial solution is a constant function $\theta(s)=0$. Using \eqref{sv}, they correspond with horizontal planes of equations $z=z_0$. If $\theta$ is not constant, and after integration, we find $\theta(s)=2\cot^{-1}\,e^{s}$. Now an integration of \eqref{sp} yields \eqref{vh}.  See Fig. \ref{fig6}, right. Last statement is consequence of the generating curve \eqref{vh}. For example, after a vertical translation, we can assume $c_1=0$. Then the change $z=\frac12\tan^{-1} e^s$ changes the parametrization of $\alpha$ by $\alpha(z)=(\frac{2}{c_1 \sin z},z)$, $z\in (0,\pi)$. This proves that $\alpha$ is a graph on the $z$-axis and asymptotic to the horizontal lines of equations $z=0$ and $z=\pi$.
\end{proof}

\begin{figure}[hbtp]
\begin{center}
\includegraphics[width=.4\textwidth]{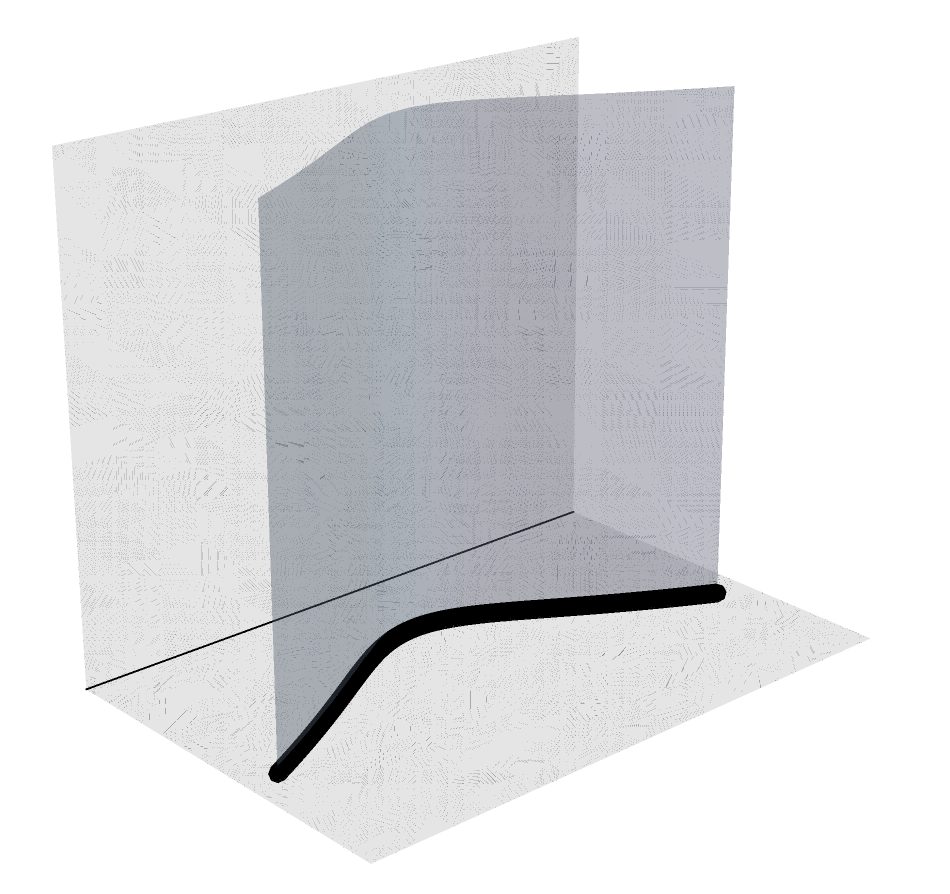}\quad \includegraphics[width=.4\textwidth]{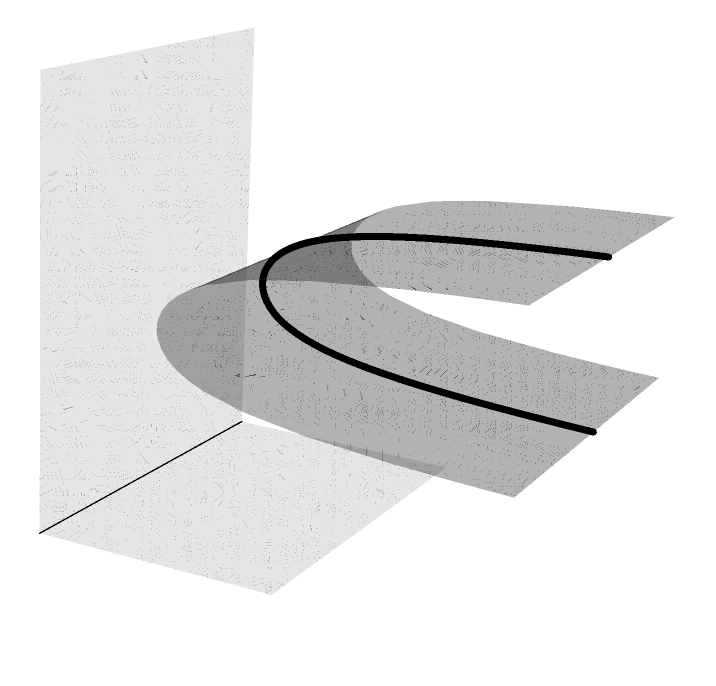}
\end{center}
\caption{Left: a vertical h-grim reaper. Right: a parabolic h-grim reaper.  }\label{fig6}
\end{figure}

\section{Grim reapers from conformal vector fields}\label{sec7}
%%%%%%%%%%%%%%%%%%%%%%%%%%%%%

In the previous sections, we have obtained the classification of all grim reapers satisfying  the translator equation $H=\langle N,X\rangle$, where $X\in\mathfrak{X}(\h^2\times\r))$ is a Killing vector field generating vertical, parabolic and hyperbolic translations of $\h^2\times\r$.

In this section we study grim reapers considering now that $X$ is   the conformal vector field $\partial_y$ and $-\partial_y$. Here we are studying solutions of equations \eqref{c+} or \eqref{c-}. The first result is when the surface is invariant by hyperbolic translations, proving that  the surfaces are trivial.

\begin{theorem}
The only hyperbolic $c_\pm$-grim reapers are the vertical plane of equation $x=0$ and the slices $\h^2\times\{z_0\}$, $z_0\in\r$.
\end{theorem}
\begin{proof}
From \eqref{normal-h} and the parametrization \eqref{p3}, we have $\langle N,\pm\partial_y\rangle=\pm e^{-t}\sin\rho\cot r$. Since the expression of $H$ in \eqref{mean-h} does not depend on the variable $t$, we deduce $H=0$ and $\cos r\sin\rho=0$. If $r=\pi/2$ then the surface is the vertical plane of equation $x=0$. If $\rho=0$, then $z(s)$ is constant, $z(s)=z_0$, and the surface is the slice $\h^2\times\{z_0\}$. 
\end{proof}

We begin by studying the vertical $c_{\pm}$-grim reapers. We will see that $c_{+}$-grim reapers  and $c_{-}$-grim reapers  are similar to  the grim reapers that are  horo-shrinkers and horo-expanders of $\h^3$, respectively. Both types of surfaces are defined as translators with respect to the conformal vector fields $\pm\partial_z$ of $\h^3$ and   were studied in \cite{bl2} and \cite{mr}. In our setting here, the generating curve is contained in $\h^2$ and we consider a vertical grim reaper, whose rulings are orthogonal to the vertical fibers of $\h^2\times\r$. Moreover, the mean curvature equation is $H=\langle N,\pm\partial_y\rangle$, with $\pm\partial_y$ being a conformal Killing vector field of $\h^3$.  This is the reason why both cases  we will simplify the arguments and we refer to \cite{bl1,mr} for further and specific details. We begin with $c_{+}$-grim reapers. 

\begin{theorem}\label{thverticalc+}
Let $\Sigma$ be a vertical $c_+$-grim reaper parametrized by \eqref{p1} such that its  generating curve $\alpha(s)=(x(s),y(s))$ satisfies \eqref{sv}. Then $\Sigma$ is one of the following examples:
\begin{enumerate}
\item The vertical plane of equation $y=2$.
\item The vertical planes of equation $x=x_0$, $x_0\in\r$.
\item A one parameter family of entire graphs $\mathcal{G}^+(y_0)$ on the $xz$-plane that are periodic along the $x$-direction. The value $y_0$ indicates the Euclidean distance of $\mathcal{G}^+(y_0)$ to the plane $y=0$. As $y_0\to 0$, then $\mathcal{G}^+(y_0)$ converges to a double covering of a vertical plane of equation $x=x_0$; if $y_0\to 2$, then $\mathcal{G}^+(y_0)$ converges to the vertical plane of equation $y=2$.
\end{enumerate}

\end{theorem}

\begin{proof} Using \eqref{c+}, \eqref{normal-v} and \eqref{mean-v}, we find that    the function $\theta$ satisfies
\begin{equation}\label{721}
\theta'=\frac{\cos\theta(2-y)}{y}.
\end{equation}
If $\cos\theta=0$, we obtain the vertical planes of equation $x=x_0$ as solutions.  From this equation,   the vertical $c_{+}$-grim reapers are determined by this equation and system \eqref{sv}. On the other hand, multiplying \eqref{721} by $\sin\theta\cos\theta$, we have a first integral, namely, $\cos\theta=c y^2e^{2/y}$ for some constant $c\in\r$. For the study of solution, it is enough to consider the autonomous system
$$\left(\begin{array}{c}y'\\ \theta'\end{array}\right)=\left(\begin{array}{c} y\sin\theta\\ \dfrac{\cos\theta(2-y)}{y}\end{array}\right).$$
  See Fig. \ref{figvertical+}, left, for a representation of the corresponding phase plane and its solutions. Notice that $(y,\theta)=(2,0)$ is an equilibrium point of this system, which corresponds with the vertical plane of equation $y=2$. From now we are in a similar situation that the study of  grim reapers that are horo-shrinkers of $\h^3$. We refer to Sect. 3 on \cite{bl2} to complete the details. See Fig. \ref{figvertical+}, right.
\end{proof}

\begin{figure}[hbtp]
\begin{center}
\includegraphics[width=.45\textwidth]{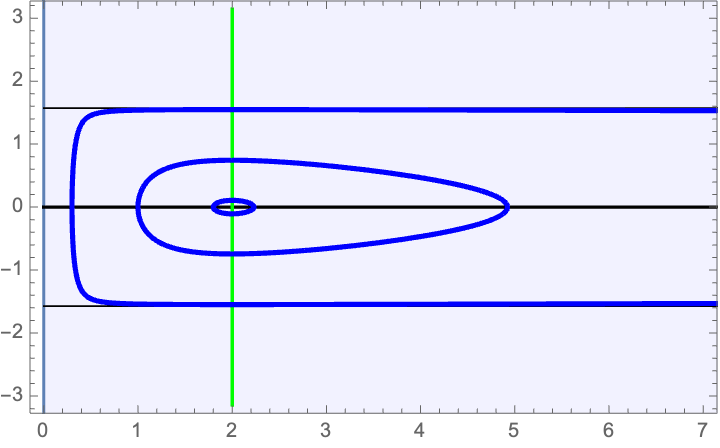}\qquad
\includegraphics[width=.4\textwidth]{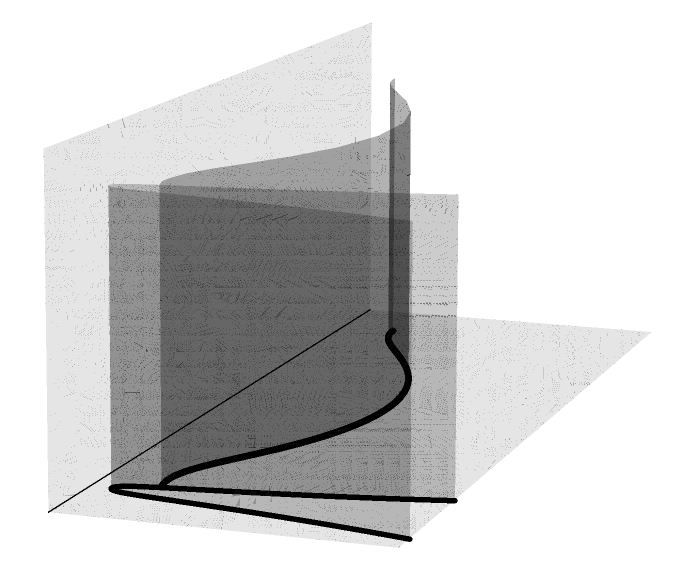}
\end{center}
\caption{Left: the phase plane of the vertical $c_+$.grim reapers. Right: generating curves and the vertical $c_+$-grim reapers.}\label{figvertical+}
\end{figure}

In case of  vertical $c_-$-grim reapers, the study is analogous to the horo-expanders that are grim reapers: see \cite{mr}.   Details are skipped and we just plot the corresponding phase plane (Fig. \ref{figverticalc-} left), and the $c_-$-grim reapers (Fig. \ref{figverticalc-}, right).
\begin{theorem}\label{thverticalc-}
Let $\Sigma$ be a vertical $c_-$-grim reaper parametrized by \eqref{p1} such that its  generating curve $\alpha(s)=(x(s),y(s))$ satisfies \eqref{sv}. Then $$
\theta'=\frac{\cos\theta(y-2)}{y}.
$$
Moreover, $\Sigma$ is one of the following examples:
\begin{enumerate}
\item The vertical plane $y=2$.
\item Vertical planes $x=x_0,\ x_0\in\r$.
\item A family of surfaces whose generating curves are connected arcs with both endpoints converging to $y=0$.
\item A family of surfaces whose generating curves are connected arcs with ope endpoint converging to $y=0$ and its other endpoint converging to $\infty$.
\item A family of surfaces whose generating curves are connected arcs with both endpoints converging to $y=\infty$.
\end{enumerate}
\end{theorem}

\begin{figure}[hbtp]
\begin{center}
\includegraphics[width=.45\textwidth]{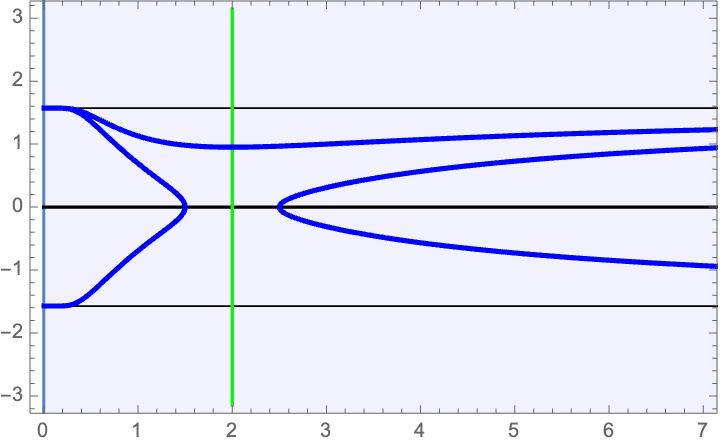}\qquad 
\includegraphics[width=.35\textwidth]{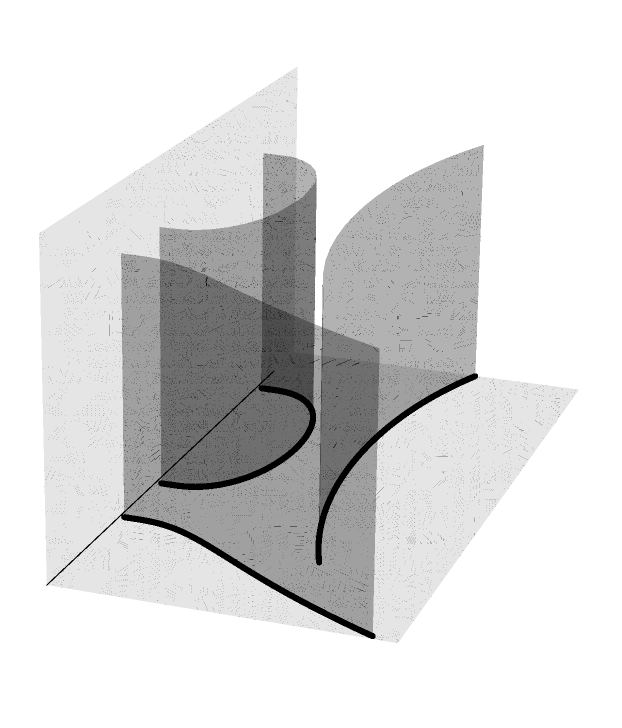}
\end{center}
\caption{Left: the phase plane of the vertical $c_-$.grim reapers. Right: generating curves and the vertical $c_-$-grim reapers.}\label{figverticalc-}
\end{figure}

The last type of grim reapers  to consider are the parabolic $c_\pm$-grim reapers. This will finish the classification of grim reapers in $\h^2\times\r$. Nonetheless, we see that these grim reapers are somehow analogous to the vertical $c_\pm$-grim reapers, since the corresponding phase planes share the same properties. Figures  \ref{parabolicc+} and \ref{parabolicc-} show  the phase planes and   parabolic $c_+$-grim reapers

\begin{theorem} Let $\Sigma$ be a parabolic  $c_\pm$-grim reaper parametrized by \eqref{p2}, where the generating curve $\alpha(s)=(y(s),z(s))$ satisfies \eqref{sp}. Then the function $\theta$ satisfies
$$
\theta'=-\frac{\sin\theta(y\pm2)}{y}.
$$
In particular, the phase planes are analogous to the ones depicted in Thms. \ref{thverticalc+} and \ref{thverticalc-} and hence their solutions.
\end{theorem}

\begin{proof} From \eqref{normal-p} and the parametrization \eqref{p2}, we have $\langle N,\pm\partial_y\rangle=\pm\sin\theta/y$. Then the result follows from  the expression of the mean curvature in \eqref{mean-p} and the equations \eqref{c+} and \eqref{c-}. The corresponding phase planes are symmetric with respect to the line $\theta=\pi/2$, instead of the line $\theta=0$, but the solutions share the same properties. 

For instance, parabolic $c_+$-grim reapers oscillate between a vertical plane and a double covering of a horizontal plane, and are periodic surfaces about a discrete group of vertical translations. In the case of parabolic $c_-$-grim reapers, there are three types of generating curves: one of them with two endpoints at $y=0$; one of them with one endpoint at $y=0$ and the other at $y=\infty$; and the last with bounded distance to $y=0$ and a double point at $y=\infty$.
\end{proof}

\begin{figure}[hbtp]
\begin{center}
\includegraphics[width=.4\textwidth]{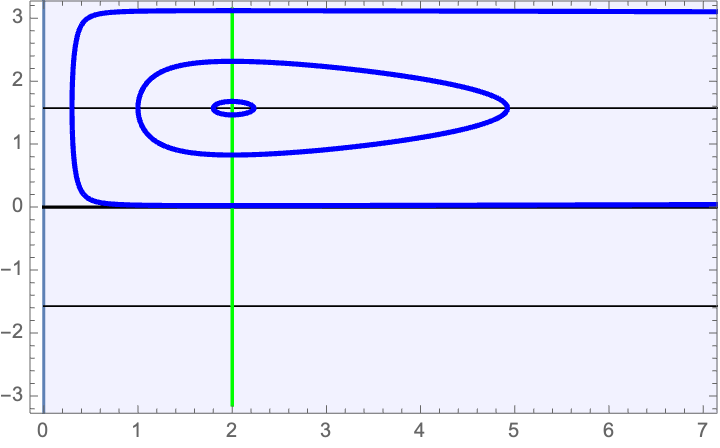}\qquad
\includegraphics[width=.25\textwidth]{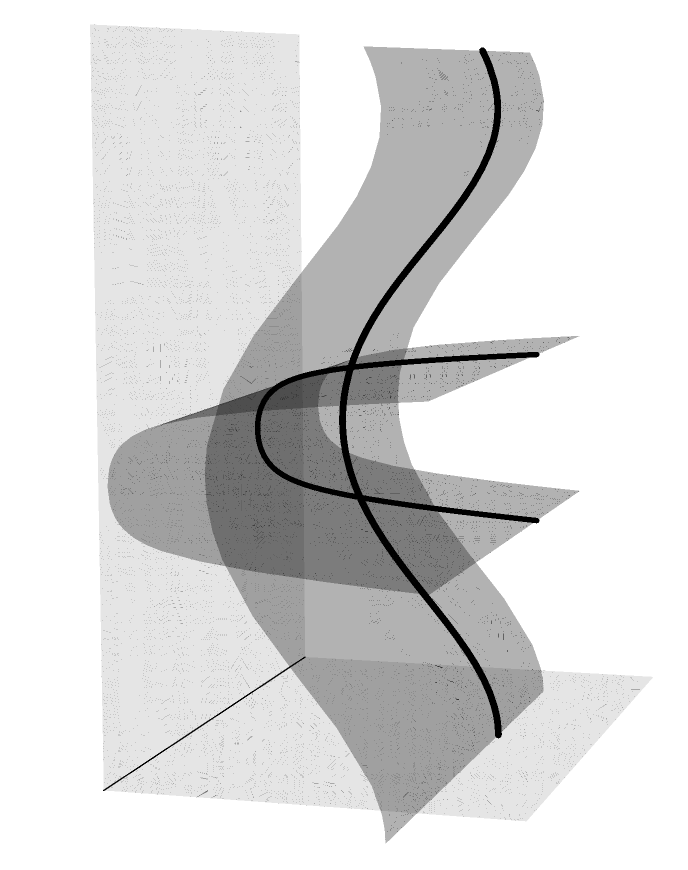}
\end{center}
\caption{Left: the phase plane of the parabolic $c_+$-grim reapers. Right: generating curves and vertical $c_+$-grim reapers.}\label{parabolicc+}
\end{figure}

\begin{figure}[hbtp]
\begin{center}
\includegraphics[width=.4\textwidth]{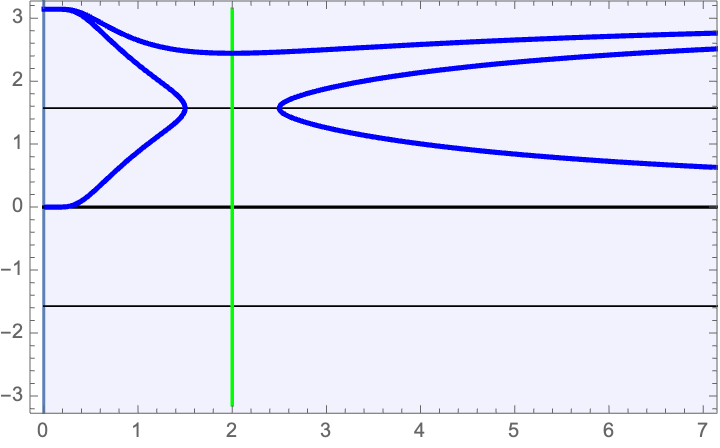}\qquad
\includegraphics[width=.35\textwidth]{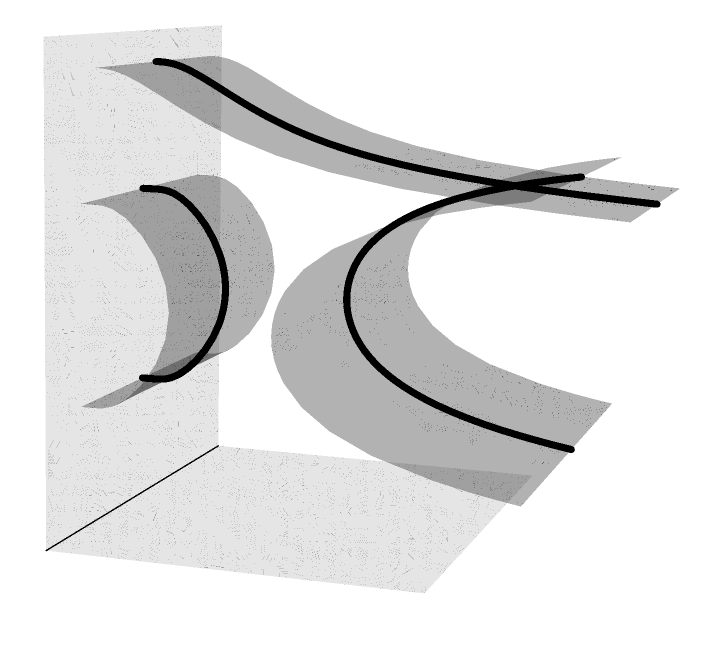}
\end{center}
\caption{Left: the phase plane of the parabolic $c_-$-grim reapers. Right: generating curves and vertical $c_-$-girm reapers.}\label{parabolicc-}
\end{figure}

%%%%%%%%%%
\section*{Acknowledgements}  

The authors want to express their gratitude to Ronaldo F. de Lima (UFRGN) and Giuseppe Pipoli (Univ.   degli Studi dell'Aquila) during the preparation of this work.   Their comments and insights have improved the paper.

\medskip

Antonio Bueno has been partially supported by CARM, Programa Regional de Fomento de la Investigaci\'{o}n, Fundaci\'{o}n S\'{e}neca-Agencia de Ciencia y Tecnolog\'{\i}a Regi\'{o}n de Murcia, reference 21937/PI/22

Rafael L\'opez  is a member of the IMAG and of the Research Group ``Problemas variacionales en geometr\'{\i}a'',  Junta de Andaluc\'{\i}a (FQM 325). This research has been partially supported by MINECO/MICINN/FEDER grant no. PID2020-117868GB-I00, and by the ``Mar\'{\i}a de Maeztu'' Excellence Unit IMAG, reference CEX2020-001105- M, funded by MCINN/AEI/10.13039/501100011033/ CEX2020-001105-M.

\end{document}